\documentclass[oneside,english]{amsart}
\usepackage{amsmath} 
\usepackage{amssymb}  
\usepackage{amsthm}
\usepackage{latexsym}
\usepackage{amsfonts,bbm}
\usepackage{xcolor}
\usepackage{mathtools}
\usepackage{enumerate}
\usepackage{cite}
\usepackage[foot]{amsaddr}
\usepackage{algorithm,algorithmic}
\usepackage[a4paper,left=2.0cm,right=2.0cm,top=2.5cm,bottom=2.5cm]{geometry}
\usepackage{multicol}

\usepackage{tikz}
\usetikzlibrary{calc,positioning, arrows.meta, shapes.geometric}
\usepackage{xargs}
\usepackage[hidelinks]{hyperref}
\usepackage{textcomp}

\usepackage{tikz}
\usetikzlibrary{positioning, arrows.meta, shapes.geometric}
\usepackage{graphicx}
\usepackage{xcolor}

\usepackage[colorinlistoftodos,prependcaption,textsize=small]{todonotes}
\usepackage{xargs}  
\usepackage{caption}
\usepackage[labelformat=simple]{subcaption}

\newcommand{\stack}[2]{\left[ \begin{smallmatrix} #1 \\ #2 \end{smallmatrix} \right]}

\newtheorem{thm}{Theorem}[section]

\newtheorem{lem}[thm]{Lemma}

\newtheorem{defn}[thm]{Definition}
\newtheorem{rem}[thm]{Remark}

\newtheorem{assumption}[thm]{Assumption}

\newcommand{\blk}{\color{black}}

\numberwithin{equation}{section}
\allowdisplaybreaks

\newcommand{\R}{\mathbb{R}}

\newcommand{\calA}{\mathcal A}
\newcommand{\calB}{\mathcal B}

\newcommand{\calE}{\mathcal E}

\newcommand{\calJ}{\mathcal J}
\newcommand{\calK}{\mathcal K}
\newcommand{\calL}{\mathcal L}

\newcommand{\calR}{\mathcal R}

\newcommand{\calS}{\mathcal S}

\newcommand{\calU}{\mathcal U}

\newcommand{\calW}{\mathcal W}
\newcommand{\calX}{\mathcal X}
\newcommand{\calY}{\mathcal Y}
\newcommand{\calZ}{\mathcal Z}

\newcommand{\frakA}{\mathfrak A}

\newcommand{\dom}[1]{D(#1)}
\newcommand{\domain}[1]{\operatorname{dom}(#1)}
\newcommand{\ran}[1]{R(#1)}

\newenvironment{smallpmatrix}
{\left(\begin{smallmatrix}}
{\end{smallmatrix}\right)}

\newenvironment{smallbmatrix}
{\left[\begin{smallmatrix}}
{\end{smallmatrix}\right]}

\title{Optimization-based control by interconnection of nonlinear port-{H}amiltonian systems}

\author{Till Preuster$^1$}\address{$^1$Junior Professorship Numerical Mathematics, Faculty of Mathematics, Chemnitz University of Technology, Germany\\ Mail: \textsc{\{till.preuster,manuel.schaller\}@math.tu-chemnitz.de}}
\author{Hannes Gernandt$^2$}\address{$^2$School of Mathematics and Natural Sciences, University of Wuppertal, Germany\\ Mail: \textsc{gernandt@uni-wuppertal.de}}
\author{Manuel Schaller$^1$}

\thanks{This work was funded by the Deutsche Forschungsgemeinschaft (DFG, German Research Foundation) – Project-ID 531152215 – CRC 1701.}

\begin{document}

\begin{abstract}
In this paper, we develop a control-by-interconnection approach for the stabilization of nonlinear port-Hamiltonian systems. Motivated by model predictive control, the controller is realized as a continuous-time primal–dual gradient flow associated with a finite-horizon, control-constrained optimal control problem. By exploiting its port-Hamiltonian structure, the optimization dynamics are interconnected with the nonlinear plant. Introducing a time-scale parameter for the optimization dynamics reveals a singularly perturbed closed-loop system, whose reduced dynamics correspond to the optimal finite-horizon feedback, while the boundary-layer dynamics capture convergence of the primal–dual variables to the optimality system. Using a composite Lyapunov function and singular-perturbation arguments, we prove asymptotic stability of the coupled plant–controller dynamics, which is local for sufficiently fast optimization dynamics and becomes global under additional assumptions. Numerical experiments for a nonlinear port-Hamiltonian oscillator illustrate the results.
\end{abstract}

\maketitle
\smallskip

\section{Introduction}
Many real-world phenomena and industrial applications can be modeled as nonlinear control systems. A central and ongoing challenge in this setting is to design control laws that simultaneously stabilize the system, optimize a prescribed performance criterion, and satisfy given constraints. In the presence of non-quadratic cost functions, nonlinear dynamics, and state or control constraints, standard approaches such as the state-feedback law derived from the solution of the algebraic Riccati equation are no longer applicable. To address these limitations, nonlinear optimization-based control schemes such as model predictive control (MPC) have emerged as well-suited alternatives. In MPC, the current state of the system is measured (or estimated), a finite-horizon optimal control problem (OCP) subject to the system dynamics and constraints is solved, and a first portion of the resulting optimal control sequence is applied to the plant, see Figure~\ref{fig:MPC} for an illustration. Under suitable assumptions, an iterative application of this procedure leads to a stabilizing controller for the plant, see the textbooks~\cite{GrunPann16,RawlMayn17}. 
\begin{figure}[htb]
    \scalebox{1.0}{\begin{tikzpicture}[
  >=Latex,
  font=\small,
  block/.style={
    draw,
    thick,
    fill=gray!15,
    rounded corners=1pt,
    inner sep=.1cm,
    align=left,
    text width=5.5cm
  },
  wire/.style={thick, -Latex},
  lab/.style={font=\small, inner sep=1pt}
]

\node[block, text width=3.5cm] (plant) {
\vspace*{-.3cm}
\begin{equation*}
\dot x_p(t) = f_p(x_p(t),u_p(t))
\end{equation*}
};

\node[block, right=-0.0cm of plant, yshift=-1.3cm] (opt) {
\vspace*{-.25cm}
\begin{align*}
\min& \quad  K(x,u)\\[-.1cm]
\mathrm{s.t.}\ \dot x(\tau) &= f(x(\tau),u(\tau)),\ x(0)=x_0\\
(x,u)&\in \mathbb{X}\times \mathbb{U}
\end{align*}
};

\coordinate (plantE) at ($(plant.east)+(0.2,0)$);
\coordinate (optN)   at ($(opt.north)+(0,0.2)$);
\coordinate (optW)   at ($(opt.west)+(-0.2,0)$);
\coordinate (plantS) at ($(plant.south)+(0,-0.2)$);

\draw[wire]
(plant.east) -- ++(2.85,0)
node[lab, above, pos=0.55]
{$x_0 = x_p(t)$}
-- (opt.north);

\draw[wire]
(opt.west) -- ++(-1.85,0)
node[lab, below, pos=0.8]
{$u_p(t)=u^*(0)$}
-- (plant.south);

\end{tikzpicture}}
\caption{Illustration of model predictive control}
\label{fig:MPC}
\vspace*{-.5cm}
\end{figure}
\noindent However, for small sampling times and in view of increasing complexity of real-world applications such as problems governed by partial differential equations, it can be challenging to compute the optimal solution in a short time. As a consequence, there is a need for MPC approaches that utilize suboptimal solutions as feedback, obtained e.g.\ by an iterative optimization algorithm. In these methods, the solution of the optimal control problem on the bottom right of Figure~\ref{fig:MPC} is replaced by an inexact solution given, e.g., by the truncation of an iterative optimization method. In recent decades, various suboptimal MPC schemes for finite-dimensional nonlinear control systems have been proposed and analyzed, see \cite[Section 8.9.2]{RawlMayn17} and \cite[Section 10.6]{GrunPann16}. In the seminal paper~\cite{GeSc_ScokMayn99}, the authors leverage non-optimal solutions for MPC with terminal ingredients, as long as a suitable Lyapunov decrease condition is satisfied. Another line of research is embedded in the context of continuation methods \cite[Section 8.9.2]{RawlMayn17} or initial-value embeddings \cite[p.~325ff]{GrunPann16} for optimization problems. One of the most prominent representatives is the real-time iteration~\cite{GeSc_DiehBock05}. There, one performs a Newton step for the optimal control problem, and it may be shown that, under suitable assumptions on the sampling time, the method yields asymptotic stability \cite{GeSc_DiehBock05}, see also the recent work \cite{Zanelli21} for a Lyapunov-based analysis of the coupled optimizer-plant dynamics. For a related perspective on continuous-time systems, see \cite{jakob2026generalized}. Third, and again related to our suggested approach, we mention instant MPC~(see \cite{YoshInou2019} for the case of linear dynamics), in which only one primal-dual gradient step is performed in the feedback computation. Despite the large body of literature regarding suboptimal MPC, the majority of the works provide feedback schemes in the absence of inequality constraints. Further, none of the works consider a function-space setting, which {is already} necessary for continuous-time optimal control problems as shown in this work and in particular is crucial to analyze suboptimal MPC for PDEs in future work, a topic {that remains largely} unexplored in the literature.


In this paper, we introduce a suboptimal MPC-like control scheme tailored to port-Hamiltonian systems (pHs). PHs \cite{maschke1993port} have emerged as a powerful modeling framework for physical systems and constitute a sub-class of dissipative systems~\cite{Brogliato07,Willems1972a}. They provide a unified energy-based representation that naturally encodes passivity, interconnection structure, and dissipation~\cite{SchJ14,jacob2012linear,MehU22}. Moreover, the port-Hamiltonian formalism offers powerful guarantees on stability and robustness under structured interconnections, making it particularly attractive for control design, e.g., via control by interconnection {(CbI)~\cite{ortega2002interconnection}}. CbI for pHs is a control paradigm in which a controller, {itself modeled} as a passive system, is interconnected with the plant in a power-preserving manner. The resulting closed-loop system inherits passivity, with its total energy given by the sum of the individual storage functions, enabling energy shaping and, in particular, stabilization.
However, achieving stabilization for pHs with input constraints remains a nontrivial task, as classical passivity-based methods typically do not address inequality constraints. 

To address this gap, we propose a CbI approach in the spirit of suboptimal model predictive control by formulating a continuous-time primal-dual gradient method for nonlinear control-constrained optimal control problems as a port-Hamiltonian system. To this end, we interconnect the dynamics of a gradient-type optimization algorithm via an MPC-type feedback with a port-Hamiltonian plant, leading to a (locally) asymptotically stable closed-loop system depicted in Figure~\ref{fig:subopt}. We provide stability results for primal-dual gradient dynamics and prove convergence of the interconnected system in function space. This is particularly crucial, as the convergence does not rely on any particular sampling or discretization scheme, as long as it preserves the port-Hamiltonian structure. Finally, the proposed analysis in function space paves the way toward suboptimal MPC for partial differential equations, a topic that has remained largely unexplored. In this context, we mention related works. Finite-dimensional saddle-point dynamics occurring in primal-dual gradient methods have been analyzed in \cite{CherMall2016,cherukuri2017saddle} and the port-Hamiltonian structure has been observed in \cite{StegPers15} for stationary finite-dimensional problems. In \cite{Pham22, vu2023port}, the authors suggest an MPC-type control-by-interconnection of pHs for discrete-time finite-dimensional optimal control problems.

\begin{figure}[htb]
    \scalebox{1.0}{\begin{tikzpicture}[
  >=Latex,
  font=\small,
  block/.style={
    draw,
    thick,
    fill=gray!15,
    rounded corners=1pt,
    inner sep=.1cm,
    align=left,
    text width=5cm
  },
  wire/.style={thick, -Latex},
  lab/.style={font=\small, inner sep=1pt}
]

\node[block] (plant) {
\vspace*{-.3cm}
\begin{align*}
\dot x_p(t) &= f_p(x_p(t)) + B_p u_p(t)\\
y_p(t) &= B_p^\top \nabla H(x_p(t))
\end{align*}
};

\node[block, right=-1cm of plant, yshift=-1.3cm] (opt) {
\vspace*{-.25cm}
\begin{align*}
\varepsilon \dot w(t) &= \calS(w(t)) + \iota_3 u_\mathrm{opt}(t)\\
y_\mathrm{opt}(t) &= \iota_3^* w(t)
\end{align*}
};

\coordinate (plantE) at ($(plant.east)+(0.2,0)$);
\coordinate (optN)   at ($(opt.north)+(0,0.2)$);
\coordinate (optW)   at ($(opt.west)+(-0.2,0)$);
\coordinate (plantS) at ($(plant.south)+(0,-0.2)$);

\draw[wire]
(plant.east) -- ++(1.6,0)
node[lab, above, pos=0.95]
{$u_\mathrm{opt}(t) = -x_p(t)$}
-- (opt.north);

\draw[wire]
(opt.west) -- ++(-1.6,0)
node[lab, below, pos=0.9]
{$u_p(t)=s(y_\mathrm{opt}(t))$}
-- (plant.south);

\end{tikzpicture}}
\caption{Proposed MPC-type CbI scheme, where $\iota_3$ is the canonical embedding of the plant state into $\calW$ and $s(\cdot)=\Pi_{[\underline{u},\overline{u}]}(-\tfrac{1}{\alpha}B_p^\top\,\cdot\,)$}
\label{fig:subopt}

\vspace*{-.2cm}
\end{figure}

A distinguishing feature of the proposed interconnection is the separation between the physical time scale of the plant and the convergence time scale of the optimization dynamics. We introduce a parameter $\varepsilon>0$ that controls the relative speed of the primal–dual dynamics. The interconnected plant–controller system then takes the singularly perturbed form depicted in Figure \ref{fig:subopt}, where $x_p$ denotes the plant state and $w$ collects the primal and dual optimization variables. In the limit $\varepsilon\to0$, the optimization variables satisfy the optimality system associated with the current plant state, and the reduced dynamics recover the corresponding finite-horizon feedback. This viewpoint makes it possible to quantify how fast the optimization dynamics must evolve relative to the plant to preserve closed-loop stability.

\noindent\textbf{Contributions and outline.}
The main contributions of this work are threefold. First, we formulate the continuous-time primal–dual gradient dynamics associated with a finite-horizon, possibly control-constrained optimal control problem as a port-Hamiltonian system on an infinite-dimensional state space. We establish well-posedness and convergence of these optimization dynamics toward the unique solution of the corresponding optimality system. Second, we interconnect the optimizer with a nonlinear port-Hamiltonian plant through a structure-preserving feedback law. By introducing a parameter $\varepsilon>0$ that determines the speed of the optimizer, we identify the resulting plant–controller dynamics as a singularly perturbed system whose reduced dynamics correspond to the optimal finite-horizon feedback. Third, using a composite Lyapunov function, we prove local asymptotic stability of the full interconnected system for all sufficiently small $\varepsilon$. Under global regularity and stability assumptions, we obtain global asymptotic stability. The analysis includes pointwise control constraints through a projected primal–dual dynamics and is independent of a particular structure-preserving discretization. Finally, we illustrate the influence of the optimizer speed and the input constraints by means of a numerical example involving a nonlinear port-Hamiltonian oscillator.

\noindent \textbf{Code availability.} The code for all numerical experiments is provided in the repository \url{https://github.com/preustertill/suboptimal_MPC.git}.

\blk
\section{Preliminaries and nonlinear pHs}
\label{sec:PHS}

This section collects the notation and operator-theoretic concepts used throughout the paper. We then introduce the class of nonlinear pHs employed to describe both the plant and the primal-dual optimization dynamics.

\subsection{Notation and operator-theoretic preliminaries}

Let $(X,\langle\cdot,\cdot\rangle_X)$ and
$(U,\langle\cdot,\cdot\rangle_U)$ be real Hilbert spaces, with induced norms
$\|\cdot\|_X$ and $\|\cdot\|_U$, respectively. We denote the space of bounded linear operators from $U$ to $X$ by $L(U,X)$. For $B\in L(U,X)$, the Hilbert space adjoint is denoted by $B^*\in {L(X,U)}$.

An operator $M:X \supset \dom{M} \to X$ is called \emph{monotone} if
\begin{equation} \label{eq:M_dissip}
    \langle M(x)-M(y),x-y\rangle_X \geq 0
\end{equation}
for all $x,y\in\dom{M}$. It is called \emph{maximally monotone}, or \emph{m-monotone}, if its graph is not properly contained in the graph of any other monotone operator. Equivalently, $M$ is maximally monotone if $ \ran{(\lambda I+M)}=X$ for some, and hence for every, $\lambda>0$; see, e.g., \cite{Barb10}. 

Let $H:X\to\mathbb R$ be Fr\'echet differentiable. {That is}, for every $x\in X$, there exists a bounded linear functional $DH(x)\in X^*$ such that
\begin{equation*}
      \lim_{\|h\|_X\to0}
    \frac{| H(x+h)-H(x)-DH(x)h|
    }{\|h\|_X}
    =0.
\end{equation*}
Using the Riesz representation theorem, we identify $DH(x)\in X^*$ with its gradient $\nabla H(x)\in X$, characterized by
\begin{equation*}
     DH(x)h=\langle\nabla H(x),h\rangle_X,
    \qquad h\in X.
\end{equation*}
For $\mu>0$, the function $H$ is called \emph{$\mu$-strongly convex} if
\begin{equation}
\label{eq:strong_convexity_H}
    H(y)
    \geq
    H(x)
    +\langle\nabla H(x),y-x\rangle_X
    +\tfrac{\mu}{2}\|y-x\|_X^2
\end{equation}
for all $x,y\in X$. In what follows, strongly convex functions that represent the stored energy of a system will be referred to as \emph{Hamiltonians}. We denote the \emph{orthogonal projection} onto a closed convex set $F \subset U$ in $U$ by $P_F:U \to F$.

\subsection{Nonlinear pHs}

We now introduce the system class used for the plant and the optimization dynamics.

\begin{defn}
\label{def:mono_phs}
Let $H:X\to\mathbb R$ be a Fr\'echet differentiable and $\mu$-strongly convex Hamiltonian. Furthermore, let $ J:X\supset\dom{J}\to X$ be a possibly nonlinear operator satisfying
\begin{equation}\label{eq:J_lossless}
    \langle J(z),z\rangle_X=0,
    \qquad z\in\dom{J},
\end{equation}
and let $R:X\to X$ be continuous and accretive, that is,
\begin{equation}\label{eq:diss}
    \langle R(z),z\rangle_X\geq0,
    \qquad z\in X.
\end{equation}
Finally, let $B\in L(U,X)$. The \emph{nonlinear port-Hamiltonian system} $(H,J,R,B)$ with an input $u$ and output $y$ is given by
\begin{subequations}
\label{eq:monotone_phs}
\begin{align}
    \dot{x}(t)
        &=
        J\bigl(\nabla H(x(t))\bigr)
        -R\bigl(\nabla H(x(t))\bigr)
        +Bu(t),
        \label{eq:monotone_phs_state}
        \\
    y(t)
        &=
        B^*\nabla H(x(t)).
        \label{eq:monotone_phs_output}
\end{align}
\end{subequations}

If, in addition, $J$ is linear and skew-adjoint and $R$ is monotone in the sense of \eqref{eq:M_dissip}, we call \eqref{eq:monotone_phs} an \emph{incrementally monotone port-Hamiltonian system}.
\end{defn}
Let $T>0$ and $u\in L^1_{\mathrm{loc}}([0,T];U)$. A function $ x:[0,T)\to X$ is called a \emph{classical solution} of \eqref{eq:monotone_phs} on $[0,T)$ corresponding to $x_0\in X$ if
\begin{enumerate}
    \item[i)] $x\in C^1([0,T);X)$;
    \item[ii)] $\nabla H(x(t)) \in \dom{J}$ for every $t\in [0,T)$;
    \item[iii)] $x(0)=x_0$;
    \item[iv)] $x(t)$ fulfills \eqref{eq:monotone_phs_state} for every $t\in [0,T)$.
\end{enumerate}
The corresponding \emph{mild solutions} are understood in the sense of \cite[Def. 11.1.3]{CurtainZwart2020}.

The properties
\eqref{eq:J_lossless} and \eqref{eq:diss} yield the standard energy inequality. Indeed, along every classical solution,
\begin{align}
    \tfrac{\mathrm d}{\mathrm dt}H(x(t))
        = 
        -
        \left\langle
            R\bigl(\nabla H(x(t))\bigr),
            \nabla H(x(t))
        \right\rangle_X
        +
        \langle y(t),u(t)\rangle_U
        \leq
        \langle y(t),u(t)\rangle_U.
\label{eq:ph_passivity}
\end{align}
Consequently, \eqref{eq:monotone_phs} is passive, i.e., it is dissipative with storage function $H$ and supply rate $\langle y,u\rangle_U$. We refer to \cite{van2004port,SchJ14} for general passivity results for nonlinear pHs.

If the underlying pHs is incrementally monotone, that is, $J$ is linear and skew-adjoint and $R$ is monotone, the dissipation inequality can also be formulated to relate two trajectories. This leads to incremental passivity with a storage function constructed from the Bregman divergence associated with $H$; see \cite{CamSch23}.

A broad range of models from physics and convex optimization fits into the framework \eqref{eq:monotone_phs}, see \cite{van2004port,ortega2002interconnection,CamSch23}  where also related nonlinear port-Hamiltonian system classes can be found. 
We briefly comment on the existence of solutions. In finite dimensions, that is, for $X=\mathbb R^n$ and $U=\mathbb R^m$, local existence of classical solutions follows from standard ordinary differential equation theory whenever $J$, $R$, and $\nabla H$ are locally Lipschitz continuous. Global existence follows under suitable growth conditions or from an a priori energy estimate derived from \eqref{eq:ph_passivity}, see e.g., \cite{esterhuizen26}. In infinite-dimensional spaces, well-posedness can be established by semigroup theory for semilinear evolution equations, see \cite[Ch. 11]{CurtainZwart2020}, whenever required.
\section{Port-Hamiltonian structures in constrained optimal control}\label{sec:optcont}
\noindent 
Before developing the control-by-interconnection scheme, we examine the dissipative structure inherent in the optimality conditions of inequality-constrained optimal control problems. To simplify notation, we abbreviate $\calY \coloneqq H^1([0,t_f],\R^n)$, $\calX \coloneqq L^2([0,t_f],\R^n)$, $\calU \coloneqq L^2([0,t_f],\R^m)$. Throughout this {paper}, we identify $\calX \cong \calX'$ and $\calU \cong \calU'$ by means of the Riesz isomorphism. For a given finite time-horizon $t_f>0$, we consider the finite-horizon optimal control problem
\begin{subequations}\label{eq:OCP}
\begin{align}
    \min_{(x,u)\in \calY \times \calU} 
    &
    K(x,u)
    \coloneqq
    \int_0^{t_f}
    \ell(x(\tau))
    + \tfrac{\alpha}{2}\|u(\tau)\|_{\R^m}^2
    \,\mathrm{d}\tau
    \label{eq:OCPa}
    \\
    \text{s.t.} \quad  &
    \mathcal{A}x - \mathcal{B}u
    =
    \stack{0}{x_0},
    \qquad
    u \in F
    \label{eq:OCPb}
\end{align}
\end{subequations}
where
\begin{itemize}
    \item $\ell:\R^n\rightarrow[0,\infty)$ is continuous and convex; 
    \item $\calA: \calX \supset  \calY \to \calX \times \R^n \eqqcolon \calZ $ encodes an ODE with system matrix $A \in \R^{n \times n}$ via
    \begin{equation*}
       \calA x \coloneqq \begin{bmatrix}\tfrac{\mathrm{d}}{\mathrm{d}\tau} x - Ax \\x(0)\end{bmatrix};
    \end{equation*}
    \item $\calB: \calU \to \calZ$ is the control operator given by
    \begin{equation}\label{eq:calB}
        \calB u = \stack{Bu}{0} \quad \text{with} \quad B \in \R^{n \times m};
    \end{equation}
    \item $\alpha > 0$ is a regularization parameter, and $x_0 \in \R^n$ is an initial value; \label{ass:init}
    \item $F$ denotes the admissible control set defined by
    \begin{align}\label{eq:F}
    F \!\coloneqq\! \left\{ u \in  \calU \, \middle| \, \underline{u}_i \!\leq\! u_i(t) \!\leq\! \overline{u}_i \ {\forall\, i,} \ \text{a.e. on} \ [0,t_f] \right\}
\end{align}
which imposes pointwise control constraints with bounds
$\underline{u},\overline{u}\in\R^m$ satisfying
$\underline{u}_i < 0 < \overline{u}_i$ for all $i=1,\dots,m$.
\end{itemize}
Throughout, we impose the following on the cost.
\begin{assumption}\label{ass:regularity_of_state_cost}
    We assume that 
    \begin{itemize}
        \item $\ell(0)=0$;
        \item the map $J:\calX \to \R$, $J(x) \coloneqq \int_0^{t_f} \ell(x(\tau))\, \mathrm{d}\tau$ admits a Fréchet derivative $\nabla J(x) \in \calX$ for all $x \in \calX$;
        \item the map {$J$} is $\mu$-strongly convex and $L$-smooth, i.e.,
        \begin{equation*}
           \hspace{-3mm} \mu \| x-y \|_\calX^2 \leq \langle \nabla J(x) - \nabla J(y), x-y \rangle_{\calX} \leq L  \| x-y \|_\calX^2
        \end{equation*}
        for all $x,y \in \calX$.
    \end{itemize}
\end{assumption}
To derive existence of optimal solutions to \eqref{eq:OCP} and associated optimality conditions, we first prove the following.
\begin{lem}\label{lem:calA}
    The operator $\calA$ admits a continuous inverse. Moreover, the Hilbert space adjoints of $\calA$ and $\calB$ are given by $\calA^*: \calX \times \R^n \supset \dom{\calA^*} \to \calX$ with
    \begin{equation*}
        \dom{\calA^*} = \left\lbrace (\lambda,\lambda_0) \in \calY \times \R^n\,\middle| \, 
             \lambda(0)=\lambda_0,  \lambda(t_f)=0
    \right\rbrace
    \end{equation*}
    and $\calB^*: \calX \times \R^n \to \calU$ with action
    \begin{equation*}
        \calA^* \stack{\lambda}{\lambda_0} =  -\tfrac{\mathrm{d}}{\mathrm{d}\tau} \lambda - A^\top \lambda, \qquad \calB^* \stack{\lambda}{\lambda_0} = B^\top \lambda.
    \end{equation*}
\end{lem}
\begin{proof}
    Let $(f,x_0) \in \calX \times \R^n$ be arbitrary. The equation
    \begin{equation*}
        \calA x = \stack{f}{x_0}
    \end{equation*}
    has a unique solution $x \in \calY$ given by
    \begin{equation*}
        x(t) \coloneqq e^{tA}x_0 + \int_0^{t} e^{(t-s)A}f(s) \, \mathrm{d}s.
    \end{equation*}
    Hence, $\calA$ defines a bijection between $\calY$ and $\calX \times \R^n$; in particular, $\calA^{-1}: \calX \times \R^n \to \calY$ exists. Since $\calA$ is closed, $\calA^{-1}$ is bounded by the bounded inverse theorem. The remaining assertions are proven in \cite[Lem. 3.2]{gernandt2025port}.
\end{proof}
Throughout this paper, we call $x \in \calY$ the \emph{primal variable} and we denote by $p \coloneqq  \stack{\lambda}{\lambda_0} \in \calZ$ the \emph{dual variable}. We define the product space $\calW\coloneqq \calX \times \calZ$ and the \emph{saddle operator}
\begin{equation*}
    \calS : \calW \supset \dom{\calA} \times  \dom{\calA^*} \to \calW
\end{equation*}
given by \begin{equation*}
    \calS (x,p)
\coloneqq 
\begin{bmatrix}
           0 & -\calA^* \vphantom{\nabla_xJ(x)}\\ \calA & 0 \vphantom{\left(\tfrac{1}{\alpha}\right)}
       \end{bmatrix}
       \begin{bmatrix}
         x  \vphantom{\nabla_xJ(x)} \\    p  \vphantom{\left(\tfrac{1}{\alpha}\right)}
    \end{bmatrix} - \begin{bmatrix}
        \nabla J(x) \\
        \calB P_F\left(\tfrac{1}{\alpha} \calB^* p\right) 
    \end{bmatrix}.
\end{equation*}
\begin{thm}\label{thm:existence_optimal_sol}
    Let $x_0 \in \R^n$ be an initial value. The optimal control problem \eqref{eq:OCP} has a unique solution $(x^\star, u^\star) \in \calY \times \calU$. Moreover, there exists a unique Lagrange multiplier $p^\star = (\lambda^\star, \lambda_0^\star) \in \dom{\calA^*}$ such that $(x^\star, p^\star)$ is a saddle point of the control-reduced \emph{Lagrangian} $\calL_\mathrm{r}: \dom{\calA} \times \dom{\calA^*} \to \R$ with
\begin{align}\label{eq:control_red_Lagrangian}
    \calL_\mathrm{r} \left(x,p \right)& \coloneqq J(x) + \langle \calA x -\stack{0}{x_0}  , p \rangle_{\calX \times \R^n} - \langle\calB^*p, P_F(\tfrac{1}{\alpha}\calB^*p) \rangle_\calU +  \tfrac{\alpha}{2}  \|P_F(\tfrac{1}{\alpha}\calB^*p)\|_\calU^2. 
\end{align}
    Equivalently, $(x^\star, p^\star)$ solves
    \begin{equation}\label{eq:optsys}
        \calS (x^\star, p^\star)=\begin{bmatrix}
        0 \\ \stack{0}{x_0} 
    \end{bmatrix}. 
    \end{equation}
    Moreover, the unique optimal control $u^\star$ fulfills 
    \begin{equation}\label{eq:optimal_control}
        u^\star = P_F\left(\tfrac{1}{\alpha} \calB^* p^\star \right) .
    \end{equation}
\end{thm}
\begin{proof}
By Lemma \ref{lem:calA}, we equivalently rewrite the optimal control problem \eqref{eq:OCP} by the (unconstrained) optimization problem
\begin{equation}\label{eq:composite}
    \min_{u \in \calU} g(u) \coloneqq K(\calA^{-1}(\calB u + \stack{0}{x_0}),u) + i_F(u)
\end{equation}
where $K$ abbreviates the cost (see \eqref{eq:OCP}) and the \emph{indicator function} $i_F: \calU \to \R \cup \{+\infty\}$ with
\begin{equation*}
    i_F(u)\coloneqq \renewcommand{\arraystretch}{0.1}
\begin{cases}
0, & \text{if } u \in F,\\
+\infty, & \text{otherwise}
\end{cases}
\end{equation*}
As $K(x,u)=J(x) + \tfrac{\alpha}{2} \| u\|_{\calU}^2$ and $J$ is convex, $\Tilde{K}(u) \coloneqq K(\calA^{-1}(\calB u + \stack{0}{x_0}),u)$ is $\alpha$-strongly convex. Moreover, since $F$ is convex, $i_F$ is also convex. Hence, \eqref{eq:composite} defines an $\alpha$-strongly convex optimization problem, which has a unique solution $u^\star \in \calU$, \cite[Prop. 2.1]{EkelTema99}. Thus, $(x^\star, u^\star) \in \calY \times \calU$ where $x^\star \coloneqq \calA^{-1}(\calB u^\star + \stack{0}{x_0})$ defines a unique solution of \eqref{eq:OCP}.

To deduce optimality conditions, observe that the full Lagrangian associated with \eqref{eq:OCP} is $\calL: \calY \times \calU \times \dom{\calA^*} \to \R$:
    \begin{equation*}
        \calL (x,u,p) \coloneqq K(x,u) + \langle \calA x - \calB u -\stack{0}{x_0}  , p \rangle_{\calX \times \R^n} + i_F(u).
    \end{equation*}
    In what follows, the superscript ${}^*$ applied to a functional denotes its Fenchel conjugate\footnote{Let $\calE$ be some Banach space and $h: \calE \to \R \cup \{\pm \infty\}$ be a function. The \emph{Fenchel conjugate} $h^* : \calE \to \R \cup \{\pm \infty\}$ of $h$ is defined by
    \begin{equation*}
         h^* (w) \coloneqq \sup_{z \in \calE} (\langle z, w \rangle_\calE - h(z)).
    \end{equation*}}. Note that for each $p \in \calX \times \R^n$:
    \begin{align*}
       \inf_{u \in \calU}  \left(\tfrac{\alpha}{2} \| u\|_{\calU}^2 + i_F(u) - \langle  u , \calB^* p \rangle_{\calU}\right) 
        &= - \sup_{u \in \calU}  \left(\langle  u , \calB^* p \rangle_{\calU}-( \tfrac{\alpha}{2} \| u\|_{\calU}^2 + i_F(u))\right) \\
        &= - \alpha \sup_{u \in \calU}  \left(\langle  u , \tfrac{1}{\alpha}\calB^* p \rangle_{\calU}-( \tfrac{1}{2} \| u\|_{\calU}^2 + i_F(u))\right) \\
        & = - \alpha( \tfrac{1}{2} \| \cdot \|_{\calU}^2 + i_F)^*(\tfrac{1}{\alpha}\calB^*p),
    \end{align*}
    where $\tfrac{1}{\alpha} i_F = i_F$ was used. By \cite[Ex. 13.5]{BausComb2011}, we conclude
    \begin{align*}
        ( \tfrac{1}{2} \| \cdot \|_{\calU}^2 + i_F)^* (v) 
       = \tfrac{1}{2} \left( \| v \|^2_\calU - d_F(v)^2 \right) = \tfrac{1}{2} \left(  \| v \|^2_\calU - \| v-P_F(v)\|_\calU^2 \right) =  \langle v, P_F(v) \rangle_\calU - \tfrac{1}{2}  \|P_F(v)\|_\calU^2 
    \end{align*}
    for all $v \in \calU$. Hence, minimizing the Lagrangian $\calL$ over $u \in \calU$ gives the reduced Lagrangian
    \begin{align*}
        \calL_\mathrm{r} (x,p) &\coloneqq \inf_{u \in \calU} \calL(x,u,p) \\
        &= J(x) + \langle \calA x -\stack{0}{x_0}  , p \rangle_{\calX \times \R^n} + \inf_{u \in \calU}  \tfrac{\alpha}{2} \| u\|_{\calU}^2 + i_F(u) - \langle  u , \calB^* p \rangle_{\calU} \\
        &= J(x) + \langle \calA x -\stack{0}{x_0}  , p \rangle_{\calX \times \R^n} - \alpha ( \tfrac{1}{2} \| \cdot \|_{\calU}^2 + i_F)^*(\tfrac{1}{\alpha}\calB^*p) \\
        &= J(x) + \langle \calA x -\stack{0}{x_0}  , p \rangle_{\calX \times \R^n} -   \langle\calB^*p, P_F(\tfrac{1}{\alpha}\calB^*p) \rangle_\calU +  \tfrac{\alpha}{2}  \|P_F(\tfrac{1}{\alpha}\calB^*p)\|_\calU^2.
    \end{align*}
    Since the Fenchel conjugate of a function is always convex,the reduced Lagrangian is concave in the dual variable. By \cite[Prop. 19.12]{BausComb2011} there exists a unique $p^\star \in \dom{\calA^*}$ such that $(x^\star, p^\star)$ is a saddle point of $\calL_\mathrm{r}$. Again by \cite[Prop. 19.12]{BausComb2011}, this is equivalent to
    \begin{equation*}
        0 = \nabla_x \calL_\mathrm{r}(x^\star, p^\star), \qquad 0 = \nabla_p \calL_\mathrm{r}(x^\star, p^\star).
    \end{equation*}
    {Clearly}, the first condition is equivalent to the first line of \eqref{eq:optsys}. By \cite[Prop. 24.4]{BausComb2011}, 
    \begin{equation*}
        \nabla_v (\tfrac{1}{2} \| \cdot \|_{\calU}^2 + i_F)^*(v)  = P_F(v)
    \end{equation*}
    for all $v \in \calU$. The chain rule for Fréchet derivatives yields
    \begin{align*}
       \nabla_p \calL_\mathrm{r}(x, p)[h]  
       &= \langle \calA x-\stack{0}{x_0}, h \rangle_{\calX \times \R^n} - \alpha  \nabla_v (\tfrac{1}{2} \| \cdot \|_{\calU}^2 + i_F)^*(\tfrac{1}{\alpha} \calB^* p)[\tfrac{1}{\alpha} \calB^* h] \\
        &=  \langle \calA x-\stack{0}{x_0}, h \rangle_{\calX \times \R^n}  - \alpha \left\langle P_F(\tfrac{1}{\alpha} \calB^* p),\tfrac{1}{\alpha} \calB^* h\right\rangle_\calU
    \end{align*}
    for all directions $h \in \calZ$. Again, by identifying the Hilbert space dual $\calZ'$ with $\calZ$ via the Riesz map we get the optimality condition
    \begin{equation*}
        0 =  \calA x^\star - \calB P_F(\tfrac{1}{\alpha} \calB^* p^\star)-\stack{0}{x_0},
    \end{equation*}
    which is the second line in \eqref{eq:optsys}. Finally, the unique minimizer in the definition of the reduced Lagrangian enjoys
\begin{align*}
    \arg\min_{u \in \calU}
    \left(
        \tfrac{\alpha}{2}\|u\|_{\calU}^2
        + i_F(u)
        - \langle u,\calB^*p^\star\rangle_{\calU}
    \right) =
    \arg\min_{u \in F}
    \tfrac{\alpha}{2}
    \left\|
        u - \tfrac{1}{\alpha}\calB^*p^\star
    \right\|_{\calU}^2 =
    P_F\left(
        \tfrac{1}{\alpha}\calB^*p^\star
    \right).
\end{align*}
Since $u^\star$ is the optimal control corresponding to the saddle point, this yields \eqref{eq:optimal_control}.
\end{proof}
We collect some important properties of the optimality operator $\calS$.
\begin{thm}\label{thm:bijectivity}
The operator $\calS$ is bijective and $\calS(0)=0$. Moreover, \(\calS^{-1} \) is Lipschitz-continuous, i.e., there exists $L_{\calS}> 0$ such that for all $h_1,h_2 \in \calW$:
\begin{equation*}
    \| \calS^{-1}(h_1)-\calS^{-1}(h_2) \|_\calW \leq L_{\calS} \| h_1-h_2 \|_\calW.
\end{equation*}
\end{thm}
\begin{proof}
Let $(x_1,p_1), (x_2,p_2) \in \dom{\calA} \times \dom{\calA^*}$ fulfill
\begin{align*}
    0= \calS (x_1,p_1) - \calS (x_2,p_2) = \begin{bmatrix}
           0 & -\calA^* \vphantom{\nabla_xJ(x)}\\ \calA & 0 \vphantom{\left(\tfrac{1}{\alpha}\right)}
       \end{bmatrix}
       \begin{bmatrix}
         x_1-x_2  \vphantom{\nabla_xJ(x)} \\    p_1-p_2  \vphantom{\left(\tfrac{1}{\alpha}\right)}
    \end{bmatrix} - \begin{bmatrix}
        \nabla J(x_1)-\nabla J(x_2) \\
        \calB P_F\left(\tfrac{1}{\alpha} \calB^* p_1\right) - \calB P_F\left(\tfrac{1}{\alpha} \calB^* p_2\right) 
    \end{bmatrix}.
\end{align*}
Testing the above with $\stack{x_1-x_2}{p_1-p_2}$ and using skew-adjointness of $\begin{smallbmatrix}
            & -\calA^* \\ \calA & 
       \end{smallbmatrix}$
       yields
\begin{align*}
    0=- \left\langle \begin{bmatrix}
        \nabla J(x_1)-\nabla J(x_2) \\
        \calB P_F\left(\tfrac{1}{\alpha} \calB^* p_1\right) - \calB P_F\left(\tfrac{1}{\alpha} \calB^* p_2\right) 
    \end{bmatrix}, \stack{x_1-x_2\vphantom{\nabla J(x_1)-\nabla J(x_2)}}{p_1-p_2\vphantom{ \calB P_F\left(\tfrac{1}{\alpha} \calB^* p_2\right)}} \right\rangle_\calW 
    \leq - \mu \| x_1 -x_2\|_\calX^2 \leq 0
\end{align*}
which proves $x_1=x_2$. Bijectivity of $\calA^*$ yields that $p_1=p_2$. This proves injectivity of $\calS$. To show surjectivity, let $(f,g) \in \calX \times \calZ$ be arbitrary. We show that there exists $(x^\star,p^\star)\in \dom{\calA} \times \dom{\calA^*}$ satisfying
\begin{align}\label{eq:system-second}
\calS(x,p) = (f,g).
\end{align}
To this end, define the functional 
\begin{equation*}
     \Phi(p) \coloneqq J^*(-f-\calA^*p)+ \alpha( \tfrac{1}{2} \| \cdot \|_{\calU}^2 + i_F)^*(\tfrac{1}{\alpha}\calB^* p)+\langle g,p\rangle_\calZ
\end{equation*}
and introduce the perturbed \emph{dual problem}
\begin{equation}\label{eq:pert_dual}
    \min_{p \in \dom{\calA^*}}\Phi(p). 
\end{equation}
Since $J$ is $L$-smooth, its Fenchel conjugate $J^*$ is $\tfrac{1}{L}$-strongly convex, \cite{BausComb2011}.
Hence, for $p_1,p_2\in \dom{\calA^*}$,
the mapping
\begin{equation*}
p\mapsto J^*(-f-\calA^*p)\eqqcolon \varphi(p)
\end{equation*}
is strongly convex because
\begin{align*}
    \varphi(p_1) &\geq \varphi(p_2) + \langle \nabla J^*(-f-\calA^*p_2), -\calA^*(p_1-p_2) \rangle_\calX   + \tfrac{1}{2L} \| \calA^*(p_1-p_2) \|_\calX^2 \\
     & \geq \varphi(p_2) + \langle \nabla \varphi(p_2), p_1-p_2 \rangle_\calZ  + \tfrac{1}{2L \|\calA^{-1} \|^2} \| p_1-p_2 \|_\calZ^2 .
\end{align*}
Because the Fenchel conjugate of a function is always convex, \cite[Prop. 13.13]{BausComb2011}, and also the linear term $p \mapsto \langle g,p\rangle_\calZ$ is convex, \eqref{eq:pert_dual} is a minimization problem governed by a strongly convex functional. Hence, there exists a unique global solution $p^\star \in \dom{\calA^*}$ of \eqref{eq:pert_dual}. Following the proof of Theorem~\ref{thm:existence_optimal_sol}, $p^\star$ fulfills the first order optimality condition 
\begin{equation*}
0 = \nabla \Phi(p^\star) = -\calA(\nabla J^*(-f-\calA^* p^\star)) + \calB P_F(\tfrac{1}{\alpha} \calB^* p^\star) +g.
\end{equation*}
Because $J$ is $\mu$-strongly convex, $\nabla J$ is bijective and
\begin{equation*}
(\nabla J)^{-1}=\nabla J^*,
\end{equation*}
\cite{BausComb2011}. Thus, \eqref{eq:system-second} is satisfied with 
\begin{equation*}
    x^\star = \nabla J^*(-f-\calA^* p^\star).
\end{equation*}
This proves the first claim since $(f,g)\in \calX \times \calZ$ was arbitrary. It is directly clear from $0 \in F$, see \eqref{eq:F}, that $P_F(0)=0$. Moreover, $\ell \geq 0$ and $\ell(0)=0$ imply that $0$ minimizes $J$, whence $\nabla J(0)=0$. Thus, $\calS(0)=0$.

To prove the Lipschitz continuity, let $ w_1= \stack{x_1}{p_1}, w_2 =  \stack{x_2}{p_2} \in \dom \calS$
    solve 
    \begin{equation*}
        \calS (w_1) = \stack{f_1}{g_1}\eqqcolon h_1,  \quad \calS (w_2) = \stack{f_2}{g_2}\eqqcolon h_2.  
    \end{equation*}
    We show that there exists $L_{\calS}> 0$ such that 
\begin{equation*}
     D \coloneqq \| w_1-w_2 \|_\calW\leq L_{\calS} \| h_1-h_2 \|_\calW \eqqcolon L_{\calS}  r .
\end{equation*}
Note that
\begin{align*}
    \| \calA^{-*} (\nabla J(x_1)-\nabla J(x_2)) \| \leq \| \calA^{-1} \| \| \nabla J(x_1)-\nabla J(x_2) \|_\calX 
    \leq L  \| \calA^{-1} \| \|  x_1-x_2 \|_\calX,
\end{align*}
 where the Lipschitz continuity of $\nabla J$ is a direct consequence of the \emph{Baillon–Haddad theorem} \cite[Cor. 18.17]{BausComb2011} and Assumption \ref{ass:regularity_of_state_cost}. We conclude
 \begin{align*}
     D&\leq \left| \left| \stack{x_1-x_2}{\calA^{-*} (\nabla J(x_1)-\nabla J(x_2))}\right|\right|_\calW +  \left| \left| \stack{0}{\calA^{-*}(f_1-f_2)}\right|\right|_\calW \\
     & \leq \sqrt{1 + L^2\| \calA^{-1} \|^2 } \|  x_1-x_2 \|_\calX + \| \calA^{-1} \| r.
 \end{align*}
 The $\mu$-strong convexity of $J$ directly gives
 \begin{align*}
     \mu \|  x_1-x_2 \|_\calX^2 & \leq - \langle \calS(w_1)-\calS(w_2) , w_1-w_2 \rangle_\calW 
     \leq \|\calS(w_1)-\calS(w_2) \|_\calW \| w_1-w_2 \|_\calW =rD
 \end{align*}
 and with that: $\|  x_1-x_2 \|_\calX \leq \tfrac{1}{\sqrt{\mu}}\sqrt{rD}$. Hence, 
 \begin{equation*}
     D \leq \sqrt{\tfrac{1 + L^2\| \calA^{-1} \|^2}{\mu}} \sqrt{rD} + \| \calA^{-1} \| r \eqqcolon c \sqrt{rD} + \| \calA^{-1} \| r.
 \end{equation*}
 Now, \emph{Young's inequality} yields $c\sqrt{r} \sqrt{D} \leq \tfrac{c^2}{2}r+\tfrac{1}{2}D$. By a combination of the above equations we get
\begin{equation*}
    D \leq (\tfrac{c^2}{2}+\| \calA^{-1} \|) r+\tfrac{1}{2}D,
\end{equation*}
which proves the claim with $L_\calS = c^2+2\| \calA^{-1} \|$.
\end{proof}
\blk
\section{Well-posedness and stability of open-loop primal-dual gradient dynamics} \label{sec:open_loop}
\noindent In this part, we formulate a primal-dual gradient method to iteratively solve the optimal control problem~\eqref{eq:OCP} or equivalently the optimality conditions \eqref{eq:optsys}. Abstractly, performing gradient ascent in the dual variable and gradient descent in the primal variable of the {control-reduced} Lagrangian $\calL_\mathrm{r}$ as in \eqref{eq:control_red_Lagrangian}, we introduce the optimizer dynamics
 \begin{align}\label{eq:graddyn}
   \begin{bmatrix}
        \hspace{-0.3mm}\dot x(t)  \vphantom{\nabla_xJ(x(t))}\hspace{-0.3mm} \\ \hspace{-0.3mm} \dot  p(t) \hspace{-0.3mm} \vphantom{\left(\tfrac{1}{\alpha}\right)}
    \end{bmatrix} \hspace{-0.95mm}  =\hspace{-0.95mm}    
    \underbrace{\begin{bmatrix}
           0 & -\calA^* \vphantom{\nabla_xJ(x(t))}\\ \calA & 0 \vphantom{\left(\tfrac{1}{\alpha}\right)}
       \end{bmatrix}}_{\eqqcolon \calJ }\hspace{-0.95mm}   
       \begin{bmatrix}
        \hspace{-0.3mm} x(t)  \vphantom{\nabla_xJ(x(t))}\hspace{-0.3mm} \\ \hspace{-0.3mm}   p(t) \hspace{-0.3mm} \vphantom{\left(\tfrac{1}{\alpha}\right)}
    \end{bmatrix} \hspace{-0.95mm}   - \hspace{-0.95mm}    \underbrace{\begin{bmatrix}
        \nabla J(x(t)) \\
       \hspace{-0.3mm} \calB P_F\left(\tfrac{1}{\alpha} \calB^* p(t)\right) \hspace{-0.3mm}
    \end{bmatrix}}_{\eqqcolon \calR(x(t),p(t))}\hspace{-0.95mm}   - \hspace{-0.95mm}   
    \begin{bmatrix}
        0 \vphantom{\nabla J(x^\star) } \\ \stack{0}{x_0} \hspace{-0.3mm}\vphantom{\left(\tfrac{1}{\alpha}\right)}
    \end{bmatrix},
\end{align}
where $x_0 \in \R^n$ is the fixed initial value of the dynamical system in \eqref{eq:OCPb}, see \ref{ass:init}. Note that 
\begin{equation*}
    \calJ : \calW \supset \dom{\calA} \times \dom{\calA^*} \to \calW, \qquad \calR : \calW \to \calW
\end{equation*}
and observe that $$ \calS = \calJ - \calR.$$

Before proving a stability result, we briefly mention that finite-dimensional saddle-point dynamics occurring in optimal control have been analyzed in the works \cite{cherukuri2017saddle,CherMall2016}. Here, as we consider the continuous-time optimal control problem, the dynamics are Hilbert space valued. Define the Hamiltonian \begin{equation*}
    H(w) =\tfrac12\|w\|_{\mathcal W}^2, \qquad w \in \calW.
\end{equation*} Then, the associated \emph{Bregman divergence} $H_{w^\star}: \calW \to \R$ is
\begin{equation}\label{eq:Lyapunov}
   H_{w^\star}(w) \coloneqq H(w) - H(w^\star) - \langle \nabla H(w^\star),w-w^\star \rangle_\calW,
\end{equation}
see \cite[Sec. 8.1]{CamSch23}, and note that, as $H$ is quadratic, the Bregman divergence simplifies to the \emph{distance to equilibrium}
 \begin{equation*}
        H_{w^\star}(w) = \tfrac{1}{2} \| w- w^\star \|_\calW^2 = H(w-w^\star).
    \end{equation*}
In what follows, we provide several properties of the saddle-point dynamics \eqref{eq:graddyn}. 
\begin{thm}\label{thm:well_posed_graddyn}
The dynamics \eqref{eq:graddyn}, that is, the system generated by $\left(H, \calJ, \calR\right)$, {defines} an autonomous incrementally monotone pH system. Moreover, for every initial condition $w_0 \in\mathcal W$, there exists a unique mild solution to \eqref{eq:graddyn}.

The system admits a unique steady state $w^\star=(x^\star,p^\star)\in \calW$, given by the unique solution of the optimality system \eqref{eq:optsys}. Furthermore, the corresponding solution satisfies \vspace{-1mm}
\begin{equation}\label{eq:energy_balance}
   H_{w^\star}(w(t)) \leq H_{w^\star}(w_0), \qquad t \ge 0. \vspace{1mm}
\end{equation}
\end{thm}
\begin{proof}
The constant term $\stack{0}{x_0}$ in \eqref{eq:graddyn} affects neither the monotonicity nor the Lipschitz estimates below, so we may set $x_0=0$ in the following computations. First, note that for $w_i=(x_i,p_i) \in \calW$, $i=1,2$:
\begin{align*}
    \langle \calR(w_1) - \calR(w_2), w_1 -w_2 \rangle_\calW  & \overset{\hphantom{(\triangledown)}}{=} \langle \nabla J(x_1) - \nabla J(x_2), x_1 -x_2 \rangle_\calX + \langle \calB P_F(\tfrac{1}{\alpha} \calB^* p_1) -  \calB P_F(\tfrac{1}{\alpha} \calB^* p_2), p_1 -p_2 \rangle_{\calZ} \\
    & \overset{(\triangledown)}{\geq} \alpha  \langle  P_F(\tfrac{1}{\alpha} \calB^* p_1) -   P_F(\tfrac{1}{\alpha} \calB^* p_2), \tfrac{1}{\alpha} \calB^* p_1 -\tfrac{1}{\alpha} \calB^*p_2 \rangle_{\calU} \\
    & \overset{(\blacktriangledown)}{\geq} \alpha  \|   P_F(\tfrac{1}{\alpha} \calB^* p_1) -   P_F(\tfrac{1}{\alpha} \calB^* p_2)\|_{\calU}^2 \geq 0,
\end{align*}
where we used convexity of the cost functional $J$ in $(\triangledown)$ and \emph{firm nonexpansiveness}\footnote{An operator $T: X \to X$ is \emph{firmly nonexpansive} if \vspace{-1mm}
\begin{equation*}
     \langle T(x)-T(y),x-y\rangle \geq \| T(x)-T(y)\|^2, \qquad \forall \, x,y \in X. \vspace{-0.4cm}
 \end{equation*}} of the projection operator $P_F: \calU \to \calU$ in $(\blacktriangledown)$, see \cite[Prop. 4.16]{BausComb2011}. This shows that $\calR$ is monotone. Since $\nabla J$ is a subgradient of a proper\footnote{A function $\phi: \calU \to \R \cup \{+\infty\}$ is called \emph{proper} if \vspace{-1mm}
    \begin{equation*} 
        \phi(x)\neq -\infty \, \text{for all } x\in \calX \qquad \text{and} \qquad \domain{\phi} \neq \emptyset. \vspace{-1mm}
    \end{equation*}
    \hspace{.3cm} with the \textit{effective domain} $
             \domain{\phi}\coloneqq \left\{u \in \calU \, \middle| \, \phi(u) < + \infty \right\}.
         $} convex lower semicontinuous
function and $p \mapsto \calB P_F\left(\tfrac{1}{\alpha} \calB^* p\right)$ is continuous and monotone, $\calR$ is a maximal monotone operator on $\calW$. Hence, $\left(H, \calJ, \calR\right)$ defines an autonomous incrementally monotone pHs. The skew-adjointness (and linearity) of $\calJ$ yields that $\calJ$ is the generator of a unitary $C_0$-group on $\calW$. Moreover, observe   \begin{align}\label{eq:Lipschitz_bound_Lyapunov}
        \| \calR(w_1)-\calR(w_2)\|_\calW  & \leq\hspace{-0.25mm}\| \hspace{-0.25mm}\nabla J(x_1)\hspace{-0.35mm}-\hspace{-0.35mm}\nabla J(x_2)\hspace{-0.25mm}\|_\calX \hspace{-0.25mm} + \hspace{-0.25mm}\| \hspace{-0.25mm}\calB P_F(\tfrac{1}{\alpha} \calB^*p_1) \hspace{-0.35mm}-\hspace{-0.35mm} \calB P_F(\tfrac{1}{\alpha}  \calB^*p_2) \hspace{-0.25mm}\|_\calZ \notag \\
        & \leq L \| x_1-x_2\|_\calX + \tfrac{1}{\alpha} \| \calB \|^2 \| p_1-p_2\|_\calZ
    \end{align}
    where we used a combination of Assumption \ref{ass:regularity_of_state_cost}, the \emph{Baillon–Haddad theorem} \cite[Cor. 18.17]{BausComb2011}, firm nonexpansiveness of the projection operator $P_F$ and boundedness of $\calB$. Thus, $\calR$ is uniformly Lipschitz continuous. By \cite[Thm. 11.1.5]{CurtainZwart2020}, for each initial value $w_0 \in \calW$, there exists a unique mild solution $w = w(\cdot, w_0) \in C([0, \infty); \calW)$ to $\left(H, \calJ, \calR\right)$. It is clear from Theorem \ref{thm:existence_optimal_sol} that the unique steady state of \eqref{eq:graddyn} is given by the unique solution $w^\star =(x^\star, p^\star)$ of \eqref{eq:optsys}. In particular, from \cite[Prop. 4.2]{Barb10} and the fact that $w(t,w^\star) \equiv w^\star$ we conclude
\begin{equation*}
    \| w(t) - w^\star \|_\calW \leq  \| w_0 - w^\star \|_\calW, \qquad t \geq 0,
\end{equation*}
and thus \eqref{eq:energy_balance}.
\end{proof}
\blk 
\subsection{Global exponential stability of primal-dual gradient dynamics}\label{subsec:exp_stab}
 \noindent In what follows, we study the asymptotic behavior of the unique mild solution of \eqref{eq:graddyn}. In the case without control constraints, i.e., if $F=\calU$, and $\mu$-strongly convex state costs $J$, the asymptotic convergence of \eqref{eq:graddyn} towards the optimal point as in Theorem \ref{thm:existence_optimal_sol} was shown in \cite[Cor. 3.5]{gernandt2025port}. We now extend this result by including inequality constraints and further improve the stability assertion by showing global uniform exponential stability of \eqref{eq:graddyn} towards the unique solution of \eqref{eq:optsys}. To this end, we leverage a Lyapunov function strongly inspired by hypocoercivity theory  \cite{villani2009hypocoercivity,achleitner2025hypocoercivity}. To this end, consider the semilinear homogeneous initial value problem
\begin{equation}\label{eq:mono_cauchy}
    \tfrac{\mathrm{d}}{\mathrm{d}t} z(t) = \frakA z(t) + f(z(t)), \qquad z(0)=z_0,
\end{equation}
where $\frakA: Z \supset \dom{\frakA} \to Z$ is the infinitesimal generator of a $C_0$-semigroup on the Hilbert space $Z$ and $f: Z \to Z$ is a possibly nonlinear uniformly Lipschitz map. By a simple shift of the state space, we assume without loss of generality that $z^\star=0$ is an equilibrium of \eqref{eq:mono_cauchy} and $f(0)=0$. By \cite[Thm. 11.1.5]{CurtainZwart2020}, for each initial value $z_0 \in Z$, there exists a unique mild solution $z = z(\cdot, z_0) \in C([0, \infty); Z)$ to \eqref{eq:mono_cauchy}. A continuous functional $V: Z \to \R_+$ is a \emph{Lyapunov functional} for \eqref{eq:mono_cauchy} on $Z$ if $V(z(t;z_0))$ is Dini differentiable at $t=0$ for all $z_0 \in Z$ and there holds
\begin{align*}
    \dot{V}_+(z_0) \coloneqq \limsup_{t \downarrow 0} \tfrac{1}{t}(V(z(t,z_0))-V(z_0)) \leq 0.
\end{align*} 
\begin{thm}\label{thm:main1}
     Let $w^\star = (x^\star,p^\star)$ be the unique solution of the optimality system \eqref{eq:optsys} and let Assumption \ref{ass:regularity_of_state_cost} be fulfilled. Then, $w^\star$ is globally uniformly exponentially stable for \eqref{eq:graddyn}, i.e., there exist $M \geq 1$ and $\eta_\mathrm{opt}> 0$ such that for each initial condition $w_0 \in \calW$:
    \begin{align*}
        \| w(t,w_0) - w^\star \|_\calW \leq M e^{-\eta_\mathrm{opt} t}\| w_0 - w^\star \|_\calW, \qquad t \geq 0.
    \end{align*}
\end{thm}
\begin{proof}
    First, we shift the state space, as outlined in \cite[Def. 11.2.1]{CurtainZwart2020}, so that the origin becomes an equilibrium of the resulting dynamical system. To this end, define $\tilde{x} \coloneqq x-x^\star, \tilde{p}\coloneqq p-p^\star, \tilde{w}\coloneqq (\tilde{x}, \tilde{p})$ and observe that
    \begin{equation}\label{eq:shifted_graddyn}
        \tfrac{\mathrm{d}}{\mathrm{d}t} \tilde{w} 
        = \calJ\tilde{w} - (\calR(\tilde{w}+w^\star) - \calR(w^\star))
    \end{equation}
    where we used that $w^\star$ fulfills the steady state equation of \eqref{eq:graddyn}. Hence, the origin is globally exponentially stable for \eqref{eq:shifted_graddyn} if and only if $w^\star$ is globally exponentially stable for \eqref{eq:graddyn}. The dynamical system \eqref{eq:shifted_graddyn} can be viewed as an instance of \eqref{eq:mono_cauchy}, where $\frakA \coloneqq \calJ$ is the generator of a unitary group on $\calW$ and
    \begin{equation*}
        f(\tilde{w}) \coloneqq - (\calR(\tilde{w}+w^\star) - \calR(w^\star)),
    \end{equation*}
    where uniform Lipschitz continuity of $f$ follows from \eqref{eq:Lipschitz_bound_Lyapunov}.

    Since $\calA: \calX \supset \calY \to \calZ$ is a closed, densely defined bijection, $\calA^{-1} :\calZ \to \calX$ is bounded. Throughout this work, we abbreviate $\| \calA^{-1} \|_{L(\calZ, \calX)}$, $\| \calB \|_{L(\calU, \calZ)}$ by $\| \calA^{-1} \|$, $\| \calB \|$, respectively. Define $W_\rho: \calW \to \R$ by 
    \begin{equation}\label{eq:Wrho}
        W_\rho(x,p) \coloneqq \tfrac{1}{2} \left \langle \begin{smallbmatrix}
            I & \rho \calA^{-1}  \\ \rho \calA^{-*} & I
        \end{smallbmatrix} \hspace{-1mm} \stack{x\vphantom{ \rho \calA^{-1}}}{p\vphantom{ \rho \calA^{-*}}}, \stack{x\vphantom{ \rho \calA^{-1}}}{p\vphantom{ \rho \calA^{-*}}}  \right\rangle_\calW.
    \end{equation}
    By the Cauchy-Schwarz inequality and boundedness of $\calA^{-1}$:
    \begin{align*}
       \rho \langle  x, \calA^{-1} p \rangle_\calX \leq \rho\| \calA^{-1} \| \| x\|_\calX\| p\|_\calZ 
        \leq \tfrac{\rho\| \calA^{-1} \|}{2}(\| x\|_\calX^2 + \| p\|_\calZ^2) 
    \end{align*}
    for all $(x,p) \in \calW$. Moreover,
    \begin{align*}
       \rho \langle x, \calA^{-1}p \rangle_\calX \geq  -\rho | \langle x, \calA^{-1}p \rangle_\calX |
        \geq -\rho\| \calA^{-1} \| \| x\|_\calX\| p\|_\calZ 
        \geq -\tfrac{\rho\| \calA^{-1} \|}{2}(\| x\|_\calX^2 + \| p\|_\calZ^2) 
    \end{align*}
    for all $(x,p) \in \calW$. Thus, for $\rho > 0$ such that $\rho \| \calA^{-1} \| < 1$ we obtain for all $w \in \calW$:
    \begin{equation}\label{eq:Lyap_function_1}
        \tfrac{1- \rho\| \calA^{-1} \|}{2} \| w \|_\calW^2 \leq W_\rho(w) \leq  \tfrac{1+ \rho\| \calA^{-1} \|}{2} \| w \|_\calW^2 .
    \end{equation}
    Clearly, $W_\rho$ is Fréchet differentiable with gradient
    \begin{equation*}
        \nabla W_\rho (w) = \begin{smallbmatrix}
            I & \rho \calA^{-1}  \\ \rho \calA^{-*} & I
        \end{smallbmatrix}w,    \qquad w \in \calW.
    \end{equation*}
    Similar to \cite[Lem. 11.2.5]{CurtainZwart2020}, for $w \in \dom{\calJ}$:
    \begin{align*}
         \langle \tilde{w}, \calJ\tilde{w} - (\calR(\tilde{w}+w^\star) - \calR(w^\star)) \rangle_\calW 
        &= - \langle w - w^\star, \calR(w) - \calR(w^\star) \rangle_\calW \\
        &= - \langle\nabla J(x) - \nabla J(x^\star), x-x^\star \rangle_{\calX} \\
        & \quad - \langle P_F(\tfrac{1}{\alpha} \calB^* p) -  P_F(\tfrac{1}{\alpha} \calB^* p^\star) , \calB^* p- \calB^* p^\star \rangle_\calU \\
        &\leq - \mu \| x-x^\star \|_{\calX}^2 - \alpha \| P_F(\tfrac{1}{\alpha} \calB^* p) -  P_F(\tfrac{1}{\alpha} \calB^* p^\star)\|_\calU^2 \\
        &\leq -  \mu \| \tilde{x} \|_{\calX}^2 .
    \end{align*}
    Observe that
    \begin{align*}
        \left \langle \begin{smallbmatrix}
            0 & \rho \calA^{-1}  \\ \rho \calA^{-*} & 0
        \end{smallbmatrix} \hspace{-1mm} \stack{\tilde{x}\vphantom{ \rho \calA^{-1}}}{\tilde{p}\vphantom{ \rho \calA^{-*}}}, f(\tilde{w}) \right\rangle_\calW 
        & =  \langle \rho \calA^{-1} \tilde{ p} , \nabla J (\tilde{x} +x^\star) - \nabla J (x^\star) \rangle_\calX  + \langle \rho \calA^{-*} \tilde{ x} , \calB P_F(\tfrac{1}{\alpha} \calB^*( \tilde{p} +p^\star)) - \calB P_F(\tfrac{1}{\alpha}  \calB^*p^\star)  \rangle_\calZ \\
        &\leq  \rho  \underbrace{(L \| \calA^{-1}\| + \tfrac{1}{\alpha}  \| \calA^{-*}\| \| \calB \|^2)}_{\eqqcolon b} \| \tilde{x} \|_\calX   \| \tilde{p} \|_\calZ   .
    \end{align*}
    Consequently,
    \begin{align*}
        \left \langle \begin{smallbmatrix}
            0 & \rho \calA^{-1}  \\ \rho \calA^{-*} & 0
        \end{smallbmatrix} \hspace{-1mm} \stack{\tilde{x}\vphantom{ \rho \calA^{-1}}}{\tilde{p}\vphantom{ \rho \calA^{-*}}}, \calJ\tilde{w} - (\calR(\tilde{w}+w^\star) - \calR(w^\star))  \right\rangle_\calW 
        &= \left \langle \begin{smallbmatrix}
            0 & \rho \calA^{-1}  \\ \rho \calA^{-*} & 0
        \end{smallbmatrix} \hspace{-1mm} \stack{\tilde{x}\vphantom{ \rho \calA^{-1}}}{\tilde{p}\vphantom{ \rho \calA^{-*}}}, \begin{smallbmatrix}
            0 & -\calA^* \vphantom{\rho \calA^{-1}} \\ \calA & 0\vphantom{\rho \calA^{-*}}
        \end{smallbmatrix} \hspace{-1mm} \stack{\tilde{x}\vphantom{ \rho \calA^{-1}}}{\tilde{p}\vphantom{ \rho \calA^{-*}}} - f(\tilde{w}) \right\rangle_\calW \\     
       &\leq \rho \| \tilde{x}\|_\calX^2 - \rho \|\tilde{p}\|_\calZ^2 + \rho b \| \tilde{x} \|_\calX   \| \tilde{p} \|_\calZ .
    \end{align*}
    Combining the above estimates, 
    \begin{align*}
        \left \langle \nabla W_\rho (\tilde{w}) , \calJ\tilde{w} - (\calR(\tilde{w}+w^\star) - \calR(w^\star))  \right\rangle_\calW  
        &\leq -( \mu -\rho) \| \tilde{x} \|_{\calX}^2 - \rho \|\tilde{p}\|_\calZ^2 + \rho b \|\tilde{x}\|_\calX\|\tilde{p}\|_\calZ  \\
        &= - \stack{\| \tilde{x} \|_\calX \vspace{0.8mm}}{\| \tilde{p} \|_\calZ}^\top \underbrace{\begin{bmatrix}
            \mu -\rho &  -\tfrac{\rho }{2} b\\ -\tfrac{\rho }{2}b & \rho 
        \end{bmatrix}}_{\eqqcolon \Lambda_\rho \in \R^{2 \times 2}} \stack{\| \tilde{x} \|_\calX \vspace{0.8mm}}{\| \tilde{p} \|_\calZ}.
    \end{align*}
    \emph{Sylvester's criterion} now yields that $\Lambda_\rho$ is positive definite iff
    \begin{equation*}
         \mu -\rho > 0 \qquad \text{and} \qquad \operatorname{det} \Lambda_\rho > 0.
    \end{equation*}
    Observe that 
    \begin{equation*}
        \operatorname{det} \Lambda_\rho = (  \mu -\rho)\rho - \rho^2 \tfrac{b^2}{4}>0 \quad \text{iff} \quad  \rho < \tfrac{\mu }{1+ \tfrac{b^2}{4}} < \mu .
    \end{equation*}
Hence, there exists a $\rho > 0$ satisfying both $\rho\|\calA^{-1}\|<1$ and $\rho < \mu/(1+\tfrac{b^2}{4})$, for which $\Lambda_\rho$ is positive definite. We abbreviate
    \begin{equation}\label{eq:mrho_Mrho}
        \begin{aligned}
        \zeta &\coloneqq \lambda_{\min}(\Lambda_\rho) =  \tfrac{\mu-\sqrt{(\mu-2\rho)^2+\rho^2 b^2}}{2}>0, \\
        m_\rho &\coloneqq \tfrac{1-\rho\|\calA^{-1}\|}{2}>0, \quad M_\rho \coloneqq \tfrac{1+\rho\|\calA^{-1}\|}{2},
        \end{aligned}
    \end{equation}
    so that \eqref{eq:Lyap_function_1} reads $m_\rho\|w\|_\calW^2 \leq W_\rho(w) \leq M_\rho \|w\|_\calW^2$ for all $w \in \calW$. A combination of \cite[Lem. 11.2.5]{CurtainZwart2020} with \eqref{eq:Lyap_function_1} yields, with $\eta \coloneqq \zeta/M_\rho>0$,
    \begin{equation}\label{eq:Lyap_function_2}
        \dot W_{\rho,+}(\tilde{w}) \leq - \zeta \| \tilde{w}\|_\calW^2 \leq - \eta\, W_\rho(\tilde{w}), \qquad \tilde w \in \calW.
    \end{equation}
    Following \cite[Thm. 11.2.7]{CurtainZwart2020}, this proves the assertion with
   \begin{equation*}
       M = \sqrt{\tfrac{M_\rho}{m_\rho}}=\sqrt{\tfrac{1+\rho\|\calA^{-1}\|}{1-\rho\|\calA^{-1}\|}}>1, \qquad \eta_\mathrm{opt} = \tfrac{\eta}{2} = \tfrac{\zeta}{1+\rho\|\calA^{-1}\|}.
   \end{equation*}
\end{proof}
\section{Closed-loop: Suboptimal MPC controller}\label{sec:CBI}
\noindent
In this section, we introduce a control-by-interconnection scheme that stabilizes the nonlinear port-Hamiltonian plant. Let the plant be given by a nonlinear port-Hamiltonian system
\begin{equation}
    \dot{x}_p = J(\nabla H(x_p)) - R(\nabla H(x_p)) + B_p u_p, \label{eq:plant}
\end{equation}
with $J$ lossless, $R$ continuous and dissipative, $B_p\in \R^{n \times m}$ and $H:\R^n\to\R$ a Hamiltonian function. Moreover, the maps $J,R, \nabla H, \kappa$ are locally Lipschitz and the input is box-constrained, $u_p(t)\in[\underline u,
\overline u]$. Define $f_p:\R^n \to \R^n$ with
\begin{equation*}
    f_p(x_p) \coloneqq  J(\nabla H(x_p)) - R(\nabla H(x_p)) .
\end{equation*} In view of Theorem \ref{thm:existence_optimal_sol}, there exists a mapping $\hat{\kappa} : \R^n \to \calU$ that assigns to each initial value $x_0$ the corresponding optimal control $u^\star = u^\star(x_0)$. Define the receding-horizon feedback as the value of the optimal control at $t=0$ by
\begin{equation*}
    \kappa(x_0) \coloneqq \hat{\kappa} (x_0)(0)=u^\star(x_0)(0).
\end{equation*}
We impose the following on the closed-loop MPC system.
\begin{assumption}\label{ass:stabilizing_controller}
Assume that there exists some $\delta>0$ such that the feedback 
\begin{equation}\label{eq:MPC_feedback}
    u_p = \kappa(x_p)
\end{equation}
renders the origin of the closed-loop system
\begin{equation}\label{eq:plant_closed_loop}
\dot{x}_p = f_p(x_p)+B_p\kappa(x_p),
\end{equation}
locally asymptotically stable, with $\calB_\delta\coloneqq\{x_p\in\R^n:
\|x_p\|_{\R^n}<\delta\}$ contained in its region of attraction. Precisely, \eqref{eq:plant_closed_loop} admits a continuously differentiable Lyapunov function $\vartheta:\calB_\delta\to\R_{\geq0}$ with $\vartheta(0)=0$ together with functions $r_1,r_p \in \calK$ (defined on $[0,\delta)$)\footnote{A continuous function $r:[0,a)\to\R_+$ is of \emph{class $\calK$}, written $r\in\calK$, if it is strictly increasing and
$r(0)=0$; if in addition $a=\infty$ and $r(s)\to\infty$ as $s\to\infty$, then $r\in\calK_\infty$.} satisfying $r_p(s)\geq\underline c\,s^2$ on $[0,\delta)$ for some $\underline c>0$, and a constant $\zeta_1>0$ such that, writing $\psi_p\coloneqq r_p(\|\cdot\|_{\R^n})$, for all $x_p\in\calB_\delta$:
\begin{subequations}\label{eq:plant_lyap}
\begin{align}
    \vartheta(x_p) &\geq r_1(\|x_p\|_{\R^n}), \label{eq:value_function_lb}\\
    \nabla\vartheta(x_p)^\top\bigl(f_p(x_p)+B_p\kappa(x_p)\bigr)
        &\leq -\psi_p(x_p), \label{eq:value_function_decay}\\
    \|B_p^\top\nabla\vartheta(x_p)\|_{\R^m}^2 &\leq \zeta_1\,\psi_p(x_p).
        \label{eq:grad_bound}
\end{align}
\end{subequations}
\end{assumption}
Denote by $\iota_3 : \R^n \to \calW$ the embedding $\iota_3 x_p = (0, (0,x_p))\top$.
\begin{rem}\label{rem:value_function}
Assume that $\vartheta$ denotes the finite-horizon value function of
\eqref{eq:OCP}. Then a standard adjoint sensitivity argument yields $\nabla\vartheta(x_0)=-\lambda_0^\star(x_0)
=-\iota_3^*\calS^{-1}(\iota_3 x_0)$, so $\nabla\vartheta(0)=0$; together with $f_p(0)+B_p\kappa(0)=0$, condition \eqref{eq:grad_bound} is a compatibility
condition tying the sensitivity of the value function \emph{in the input
directions} to the decay rate $\psi_p=r_p(\|\cdot\|_{\R^n})$: the port gradient
$B_p^\top\nabla\vartheta$ must vanish at least as fast as $\psi_p^{1/2}$ as
$x_p\to 0$.
\end{rem}
The pH structure of the plant yields the following local bound. 
\begin{lem}\label{lem:field_bound_ph}
Under Assumption~\ref{ass:stabilizing_controller} there exists $\zeta_p>0$ s.t.
\begin{equation}\label{eq:field_bound}
    \|f_p(x_p)+B_p\kappa(x_p)\|_{\R^n}^2 \leq \zeta_p\,\psi_p(x_p)
\end{equation}
for all $x_p\in\calB_\delta$.
\end{lem}
\begin{proof}
On $\overline{\calB_\delta}$ the map $x_p\mapsto f_p(x_p)+B_p\kappa(x_p)$ is Lipschitz and vanishes at
the origin, hence $\|f_p(x_p)+B_p\kappa(x_p)\|_{\R^n}\leq L_\delta\|x_p\|_{\R^n}$
for some $L_\delta>0$. Using the quadratic lower bound $r_p(s)\geq\underline c
\,s^2$ together with $\psi_p(x_p)=r_p(\|x_p\|_{\R^n})$ we conclude
\begin{align*}
    \|f_p(x_p)+B_p\kappa(x_p)\|_{\R^n}^2 &\leq L_\delta^2\|x_p\|_{\R^n}^2 \leq \tfrac{L_\delta^2}{\underline c}\psi_p(x_p) ,
\end{align*}
which proves the assertion with $\zeta_p \coloneqq L_\delta^2/\underline c$.
\end{proof}
We now explain the main idea of the controller design. In the definition of $\calB$, see \eqref{eq:calB}, we choose $B=B_p$. By Theorem \ref{thm:existence_optimal_sol}, for a fixed $x_0 \in \R^n$, the resulting optimal control $u^\star$ fulfills 
\begin{equation*}
    u^\star = P_F(\tfrac{1}{\alpha}B_p^\top \lambda^\star).
\end{equation*}
It is easily verified that $P_F:\calU \to \calU$ acts pointwise:
\begin{equation}\label{eq:orth_proj}
    P_F(u)(t)
    =
    \Pi_{[\underline{u},\overline{u}]}(u(t))
    \qquad\text{for a.e. }t\in[0,t_f],
\end{equation}
where $\Pi_{[\underline{u},\overline{u}]}:\mathbb{R}^m\to
[\underline{u},\overline{u}]$ denotes the Euclidean projection onto the box
$[\underline{u},\overline{u}]$ and is given componentwise by
\begin{equation*}
    \bigl(\Pi_{[\underline{u},\overline{u}]}(v)\bigr)_i
    =
    \min\!\left\{\overline{u}_i,
    \max\!\left\{\underline{u}_i,v_i\right\}\right\},
    \qquad i=1,\ldots,m.
\end{equation*}
According to \eqref{eq:orth_proj}, the state-feedback controller \eqref{eq:MPC_feedback} is equivalently rewritten as
\begin{equation*}
    u_p = \kappa(x_p) = \Pi_{[\underline{u},\overline{u}]}(\tfrac{1}{\alpha}B_p^\top\lambda_0^\star(x_p)).
\end{equation*}
Hence, under Assumption \ref{ass:stabilizing_controller}, the origin of the \emph{differential algebraic equation}
\begin{subequations}\label{eq:DAE}
    \begin{align}
    \dot x_p(t) & = f(x_p(t), w(t))\\
   0 & = g(x_p(t), w(t)) 
\end{align}
\end{subequations}
where
    \begin{align*}
   f(x_p,w) &= f_p(x_p) + B_p\Pi_{[\underline{u},\overline{u}]}(\tfrac{1}{\alpha} B_p^\top \lambda_0), \\
   g(x_p, w) &= \calJ w - \calR (w) - \iota_3 x_p,
\end{align*}
is locally asymptotically stable. 
Our goal is to interpret \eqref{eq:DAE} as the reduced dynamics of the autonomous singularly perturbed system
\begin{subequations}\label{eq:sing_pert}
    \begin{align}
    \dot x_p(t) & = f(x_p(t), w(t)) \label{eq:sing_pert1}  \\
    \varepsilon \dot w(t) & = g(x_p(t), w(t)), \label{eq:sing_pert2}
\end{align}
\end{subequations}
see also \cite{khalil2002nonlinear}. Precisely, the controller depicted in Figure \ref{fig:subopt} is given by \eqref{eq:sing_pert2}.

We now construct a Lyapunov function for the singularly perturbed dynamics \eqref{eq:sing_pert}, thereby showing existence of a global mild solution and asymptotic stability for a given set of initial conditions. 
Based on the proof of Theorem \ref{thm:main1}, we choose the Lyapunov function  $V_\rho: \R^n \times \calW \to [0, \infty)$ as
    \begin{equation*}
        V_\rho(x_p,w)\coloneqq W_\rho(w-w^\star(x_p)),
    \end{equation*} 
    where $W_\rho$ is defined in \eqref{eq:Wrho}. This quantity measures the (weighted) distance of the current optimizer state to the moving equilibrium $w^\star(x_p)$. Recall that, by \eqref{eq:Lyap_function_2},
\begin{align}\label{eq:Lyap_estimate_cl_optimizer}
     \nabla_w V_\rho (x_p, w) g(x_p,w)  \leq -\underbrace{\tfrac{2\zeta}{1+\rho \| \calA^{-1} \|}}_{\eqqcolon \eta } V_\rho(x_p,w)
\end{align}
for some constant $\zeta>0$ and a sufficiently small $\rho>0$, see \eqref{eq:Lyap_function_2}. We introduce the following candidate \emph{composite Lyapunov function} $\nu: \calB_\delta \times \calW \to [0, \infty)$ with
\begin{equation*}
    \nu(x_p,w)\coloneqq (1-d)\vartheta(x_p)+dV_\rho(x_p,w), \qquad 0<d<1.
\end{equation*}
\begin{thm}\label{thm:main2}
Let Assumption~\ref{ass:stabilizing_controller} be true. Then, there is $\varepsilon^\star >0$ such that for all $\varepsilon <\varepsilon^\star$, there exists $\delta'>0$ such that for every initial condition $z_0=(x_p^0,w_0)$ with $\|z_0\|_{\R^n \times \calW}<\delta'$,
the system \eqref{eq:sing_pert} admits a unique bounded mild solution on $[0, \infty)$, and this solution converges to zero. Consequently, the origin is locally asymptotically stable. If $f_p$ is globally Lipschitz and if Assumption \ref{ass:stabilizing_controller} holds with $\delta=\infty$ and $r_1\in\calK_\infty$, the asymptotic stability holds globally.
\end{thm}
\begin{proof}
On $Z\coloneqq\R^n\times\calW$, observe that \eqref{eq:sing_pert} admits the semilinear
form 
\begin{equation*}
    \dot z=\frakA_\varepsilon z+f_\varepsilon(z), \qquad z=(x_p,w),
\end{equation*}
with
\begin{equation*}
    \frakA_\varepsilon\coloneqq
    \begin{bmatrix} 0 & 0\\ 0 & \tfrac1\varepsilon\calJ\end{bmatrix},
    \qquad
    f_\varepsilon(z)\coloneqq
    \begin{bmatrix}
        f_p(x_p)+B_p\Pi_{[\underline u,\overline u]}
            (\tfrac1\alpha B_p^\top\lambda_0)\\[0.5mm]
        -\tfrac1\varepsilon\bigl(\calR(w)+\iota_3 x_p\bigr)
    \end{bmatrix}.
\end{equation*}
The operator $\frakA_\varepsilon$ generates the $C_0$-semigroup
\begin{equation*}
    (\operatorname{diag}(I,T_\varepsilon(t)))_{t \geq 0} \in L(Z),
\end{equation*}
where $(T_\varepsilon(t))_{t\geq0}$ is the unitary group generated by $\tfrac1\varepsilon \calJ$. The map $f_p$ is locally Lipschitz by Assumption~\ref{ass:stabilizing_controller}, $w \mapsto \lambda_0$ is linear and bounded, $\Pi_{[\underline u,\overline u]}$ is nonexpansive, $\calR$ is globally Lipschitz, and $\iota_3$ is linear and bounded. Hence, the map $f_\varepsilon$ is locally Lipschitz. Thus, by \cite[Thm.~11.1.5]{CurtainZwart2020}, there exists $t_{\max}>0$ such that \eqref{eq:sing_pert} has a unique mild solution $z=(x_p,w)$ on $[0,t_{\max})$,
and if $t_{\max}<\infty$, then
\begin{equation}\label{eq:finite_time_blow_up}
    \lim_{t\uparrow t_{\max}}\|z(t)\|_Z=\infty.
\end{equation}
We now establish
\begin{equation}\label{eq:Lyap_for_sing_pert}
    \dot \nu_+ (x_p,w) \leq 0 \qquad \text{on } \calB_\delta \times \calW
\end{equation} 
by first showing \eqref{eq:Lyap_for_sing_pert} on $\calB_\delta \times \dom{\calJ}$ and then extending the  corresponding integral inequality to arbitrary initial values in $\mathcal B_\delta\times\mathcal W$.

However, for general $(x_p, w) \in \calB_\delta \times \calW$, write $$\Delta f(x_p,w)\coloneqq f(x_p,w)-f(x_p,w^\star(x_p))$$ and recall that
$f(x_p,w^\star(x_p))=f_p(x_p)+B_p\kappa(x_p)$. Differentiating $\nu$ along the
trajectories of \eqref{eq:sing_pert} and using \eqref{eq:value_function_decay},
\begin{align}
\dot\nu_+(x_p,w)
&\leq (1-d)\nabla\vartheta(x_p)^\top f(x_p,w^\star(x_p)) + d\,\dot V_{\rho,+}(x_p,w) + (1-d)\nabla\vartheta(x_p)^\top\Delta f(x_p,w) \notag\\
&\leq -(1-d)\psi_p(x_p) + d\,\dot V_{\rho,+}(x_p,w) + (1-d)\nabla\vartheta(x_p)^\top\Delta f(x_p,w). \label{eq:nu_dot_start}
\end{align}
It remains to control the plant cross term and $\dot V_{\rho,+}$.

We first treat the plant cross term. By Cauchy-Schwarz, the nonexpansiveness of
$\Pi=\Pi_{[\underline u,\overline u]}$, \eqref{eq:grad_bound} and
\eqref{eq:Lyap_function_1},
\begin{align*}
    \nabla\vartheta(x_p)^\top\Delta f(x_p,w)
    &= \big\langle \nabla\vartheta(x_p),\,
        B_p\big(\Pi(\tfrac1\alpha B_p^\top\lambda_0)
        -\Pi(\tfrac1\alpha B_p^\top\lambda_0^\star(x_p))\big)\big\rangle \\
    &= \big\langle B_p^\top \nabla\vartheta(x_p),\,
        \Pi(\tfrac1\alpha B_p^\top\lambda_0)
        -\Pi(\tfrac1\alpha B_p^\top\lambda_0^\star(x_p))\big\rangle \\
    &\leq \tfrac{\|B_p\|}{\alpha}\,\|B_p^\top \nabla\vartheta(x_p)\|\,
        \|\lambda_0-\lambda_0^\star(x_p)\| \\
    &\leq \tfrac{\sqrt{\zeta_1}\|B_p\|}{\alpha}\,
        \psi_p^{1/2}(x_p)\,\|w-w^\star(x_p)\| \\
    &\leq \beta_1\,\psi_p^{1/2}(x_p)\,V_\rho^{1/2}(x_p,w),
\end{align*}
where
\begin{equation*}
     \beta_1\coloneqq\tfrac{\sqrt{2\zeta_1}\|B_p\|}
        {\alpha\sqrt{1-\rho\|\calA^{-1}\|}} .
\end{equation*}

Second, we show optimizer decay. Define $$Q_\rho \coloneqq \begin{bmatrix}
            I & \rho \calA^{-1}  \\ \rho \calA^{-*} & I
        \end{bmatrix} \in L(\calW).$$
We first establish the time regularity needed for the following difference-quotient calculation. Let $z_0 = (x_p^0, w_0) \in \mathcal B_\delta \times \operatorname{dom}\mathcal J$ and let $z(t) = (x_p(t), w(t))$ be the corresponding mild solution of \eqref{eq:sing_pert} on $[0,t_{\max})$.
The plant component satisfies
\begin{equation*}
    x_p(t) = x_p^0 + \int_0^t f(x_p(s), w(s)),\mathrm ds.
\end{equation*}
Hence
\begin{equation*}
    x_p \in C^1([0,t_{\max});\mathbb R^n),
    \qquad
    \dot x_p(t) = f(x_p(t), w(t)),
\end{equation*}
and in particular $x_p \in W^{1,\infty}([0,T];\mathbb R^n)$ for every $T < t_{\max}$. The optimizer equation reads
\begin{equation*}
\varepsilon \dot w(t) + \mathcal M(w(t)) = -\iota_3 x_p(t),
\qquad
\mathcal M := -\mathcal J + \mathcal R.
\end{equation*}
Moreover, the right derivative satisfies 
\begin{equation*} 
    \tfrac{\mathrm{d}^+}{\mathrm{d}t} w(t) = \tfrac{1}{\varepsilon} g(x_p(t), w(t)). 
\end{equation*}
Define
\begin{equation*}
    e(t) \coloneqq w(t) - w^\star(x_p(t)), \quad e_0 \coloneqq e(0), \quad \delta e(t) \coloneqq e(t)-e_0,
\end{equation*}
        and observe that $\nabla_w V_\rho (x_p^0, w_0)=Q_\rho e_0$.
Since $w$ is Lipschitz, $x_p$ is $C^1$, and $w^\star$ is Lipschitz,
\begin{equation*}
w(t) - w_0 = \mathcal O(t),
\qquad
w^\star(x_p(t)) - w^\star(x_p^0) = \mathcal O(t),
\end{equation*}
hence $\delta e(t) = \mathcal O(t)$ as $t \downarrow 0$.
Then, 
        \begin{align*}
             \langle Q_\rho e(t), e(t) \rangle_\calW 
            =  \langle Q_\rho (e_0 + \delta e(t)),  e_0 + \delta e(t) \rangle_\calW 
            = \langle Q_\rho e_0, e_0 \rangle_\calW + 2 \langle Q_\rho e_0, \delta e(t) \rangle_\calW  + \langle Q_\rho \delta e(t), \delta e(t) \rangle_\calW.
        \end{align*}
        Hence
        \begin{align*}
            \tfrac{1}{2}(\langle Q_\rho e(t), e(t) \rangle_\calW-\langle Q_\rho e_0, e_0 \rangle_\calW)  = \langle Q_\rho e_0, \delta e(t) \rangle_\calW  + \tfrac{1}{2} \langle Q_\rho \delta e(t), \delta e(t) \rangle_\calW.
        \end{align*}
        Using $\delta e(t) \in \mathcal{O}(t)$ and taking the limit $t \downarrow 0$ yields 
        \begin{align*}
             \limsup_{t \downarrow 0}  \tfrac{1}{2t}(\langle Q_\rho e(t), e(t) \rangle_\calW-\langle Q_\rho e_0, e_0 \rangle_\calW)  = \limsup_{t \downarrow 0} \tfrac{1}{t}\langle Q_\rho e_0, \delta e(t) \rangle_\calW.
        \end{align*}
        Using the fact
        \begin{equation*}
            \delta e(t) = (w(t)- w_0) - (w^\star(x_p(t))-w^\star(x_p^0))
        \end{equation*}
        we conclude
\begin{align*}
     \dot V_{\rho,+} (x_p^0, w_0) 
     & = \limsup_{t \downarrow 0}  \left\langle Q_\rho e_0, \tfrac{w(t)- w_0}{t} - \tfrac{w^\star(x_p(t))-w^\star(x_p^0)}{t}  \right\rangle_\calW  \\
     & = \tfrac{1}{\varepsilon}\nabla_w V_\rho (x_p^0, w_0) g(x_p^0,w_0) + \limsup_{t \downarrow 0} \left[ - \left\langle Q_\rho e_0, \tfrac{w^\star(x_p(t))-w^\star(x_p^0)}{t}  \right\rangle_\calW \right] .
\end{align*}    
By the Lipschitz continuity of $\calS^{-1}$ (Theorem~\ref{thm:bijectivity}) and
$w^\star=\calS^{-1} \circ \iota_3$, the second term fulfills
\begin{align*}
      - \left\langle Q_\rho e_0, \tfrac{w^\star(x_p(t))-w^\star(x_p^0)}{t}  \right\rangle_\calW 
     &   \leq \tfrac{1}{t} \| Q_\rho e_0 \|_\calW  \| \calS^{-1} (\iota_3 x_p(t)) -  \calS^{-1} (\iota_3 x_p^0) \|_\calW \\
     &  \leq   L_{\calS} \| Q_\rho \|_{L(\calW)} \| e_0 \|_\calW \| \tfrac{ x_p(t) -x_p^0}{t}\|_{\R^n}.
\end{align*} 
As a consequence and again using \eqref{eq:Lyap_function_1}:
\begin{align*}
     \limsup_{t \downarrow 0} \left[ - \left\langle Q_\rho e_0, \tfrac{w^\star(x_p(t))-w^\star(x_p^0)}{t}  \right\rangle_\calW \right] 
    & \leq  L_{\calS} \| Q_\rho \|_{L(\calW)} \| w_0 - w^\star(x_p^0) \|_\calW \| f(x_p^0,w_0)\|_{\R^n} \\
    & \leq  \tfrac{\sqrt{2}L_{\calS} \| Q_\rho \|_{L(\calW)}}{\sqrt{1-\rho \| \calA^{-1} \|}} V_\rho^{1/2}(x_p^0,w_0) \| f(x_p^0,w_0)\|_{\R^n}.
\end{align*}
Moreover, splitting $f$ yields
\begin{align*}
    \| f(x_p^0,w_0)\|_{\R^n} 
    &\leq  \|  f(x_p^0,w^\star(x_p^0))\|_{\R^n} + \| \Delta f(x_p^0,w_0) \|_{\R^n}  \\
    & \leq  \sqrt{\zeta_p }\,\psi_p^{1/2}(x_p^0)+ \tfrac{\sqrt{2}\| B_p \|^2}{\alpha \sqrt{1-\rho \| \calA^{-1} \|}} V_\rho^{1/2}(x_p^0,w_0),
\end{align*}
where we used \eqref{eq:field_bound} for the first summand and \eqref{eq:Lyap_function_1}
for the second. Combining everything gives
\begin{align*}
    \dot V_{\rho,+}(x_p,w)
    \leq \tfrac1\varepsilon\nabla_w V_\rho(x_p,w)g(x_p,w)+ \beta_2\,\psi_p^{1/2}(x_p)V_\rho^{1/2}(x_p,w)
        + \gamma\,V_\rho(x_p,w),
\end{align*}
with
\begin{equation*}
    \beta_2\coloneqq\tfrac{\sqrt{2\zeta_p}L_\calS\|Q_\rho\|_{L(\calW)}}
{\sqrt{1-\rho\|\calA^{-1}\|}} \quad \text{and} \quad \gamma\coloneqq\tfrac{2L_\calS\|Q_\rho\|_{L(\calW)}\|B_p\|^2}
{\alpha(1-\rho\|\calA^{-1}\|)}.
\end{equation*}
Inserting both steps into \eqref{eq:nu_dot_start} and using
$\nabla_w V_\rho(x_p,w)g(x_p,w)\le-\eta V_\rho(x_p,w)$
(see \eqref{eq:Lyap_estimate_cl_optimizer}),
\begin{align*}
\dot\nu_+(x_p,w)
& \leq -(1-d)\psi_p(x_p)
    + d\Big(\gamma-\tfrac\eta\varepsilon\Big)V_\rho(x_p,w) + \big((1-d)\beta_1+d\beta_2\big)
    \psi_p^{1/2}(x_p)V_\rho^{1/2}(x_p,w) \\
&= -\begin{bmatrix}\psi_p^{1/2}(x_p)\\
    V_\rho^{1/2}(x_p,w)\end{bmatrix}^{\!\top}
    \Xi_\varepsilon
    \begin{bmatrix}\psi_p^{1/2}(x_p)\\
    V_\rho^{1/2}(x_p,w)\end{bmatrix},
\end{align*}
where the symmetric matrix
\begin{equation*}
    \Xi_\varepsilon\coloneqq
    \begin{bmatrix}
        1-d & -\tfrac12\big((1-d)\beta_1+d\beta_2\big)\\[1mm]
        -\tfrac12\big((1-d)\beta_1+d\beta_2\big) &
        d\big(\tfrac\eta\varepsilon-\gamma\big)
    \end{bmatrix}
\end{equation*}
depends on the weight $d$ of the composite Lyapunov function $\nu$.
As $\psi_p^{1/2}(x_p),V_\rho^{1/2}(x_p,w)\ge 0$ vanish simultaneously only at
$(x_p,w)=(0,0)$, we have $\dot\nu_+<0$ off the equilibrium whenever
$\Xi_\varepsilon\succ 0$. By Sylvester's criterion this holds iff
\begin{equation*}
    d(1-d)\Big(\tfrac\eta\varepsilon-\gamma\Big)
    > \tfrac14\big((1-d)\beta_1+d\beta_2\big)^2,
\end{equation*}
i.e.\ $\varepsilon<\varepsilon_d\coloneqq
\eta\big/\big(\gamma+\tfrac{1}{4d(1-d)}((1-d)\beta_1+d\beta_2)^2\big)$.
Maximizing $\varepsilon_d$ over $d\in(0,1)$ gives $d^\star=\beta_1/(\beta_1+
\beta_2)$ and
\begin{equation*}
    \varepsilon^\star=\varepsilon_{d^\star}=\tfrac{\eta}{\gamma+\beta_1\beta_2},
\end{equation*}
cf.\ \cite[Thm.~11.3]{khalil2002nonlinear}. From now on we fix the weight $d=d^\star$ in the definition of $\nu$; all subsequent occurrences of $\nu$ and $\Xi_\varepsilon$ refer to this choice. Note that $\varepsilon^\star$ depends only on $\eta,\gamma,\beta_1,\beta_2$ and hence neither on $z_0$ nor on the auxiliary radius $\delta_r$ introduced below. For every $\varepsilon < \varepsilon^\star$ the matrix
$\Xi_\varepsilon$ is positive definite, hence
\begin{equation*}
     \dot\nu_+(x_p,w) \leq -\lambda_{\min}(\Xi_\varepsilon)\bigl(\psi_p(x_p)+V_\rho(x_p,w)\bigr)<0
\end{equation*}
for all $(x_p,w)\in \calB_\delta \times\dom \calJ$ with $(x_p,w)\neq(0,0)$.

The preceding calculation is initially valid for classical solutions. Its integral form extends to mild solutions by the standard 
ument: approximate the initial value $z_0 \in \calB_\delta \times \calW$ by elements of the dense domain $\dom{\frakA_\varepsilon}$, use continuous dependence of mild solutions on their initial values, and pass to the limit. Consequently, every mild solution that remains in $\calB_\delta\times\calW$ satisfies
\begin{align}\label{eq:nu_integral_decay}
    \hspace{-1.0mm}\nu(z(t)\hspace{-0.2mm})\hspace{-0.5mm}-\hspace{-0.5mm}\nu(z(s)\hspace{-0.2mm})
    \hspace{-0.5mm}\leq \hspace{-0.5mm} -\lambda_\varepsilon \hspace{-1mm} \int_s^t\hspace{-0.5mm}
    (\psi_p(x_p(\tau)\hspace{-0.2mm})\hspace{-0.5mm}+\hspace{-0.5mm}V_\rho(z(\tau)\hspace{-0.2mm})) \, \mathrm d\tau,
\end{align}
for $0\leq s\leq t<t_{\max}$, where
$\lambda_\varepsilon\coloneqq\lambda_{\min}(\Xi_\varepsilon)>0$. By \cite[Lem. 11.2.5]{CurtainZwart2020}, this proves \eqref{eq:Lyap_for_sing_pert}.

Now, choose an auxiliary radius $\delta_r  \in (0, \delta)$, fix
\begin{equation*}
   c \in  (0,(1-d^\star)\,r_1(\delta_r))
\end{equation*}
and define
\begin{equation*}
    \Omega_c \coloneqq \left\{z=(x_p,w)\in\calB_{\delta_r}\times\calW:\nu(z)<c\right\}.
\end{equation*}
Let $z_0\in\Omega_c$, and let $z(\cdot;z_0)=(x_p(\cdot),w(\cdot))$ be the corresponding maximal mild solution on $[0,t_{\max})$. Introduce the first exit time
\begin{equation*}
    t_r\coloneqq
    \inf\left\{t\in[0,t_{\max}):\|x_p(t)\|_{\R^n}\geq\delta_r\right\},
\end{equation*}
with the convention $\inf\emptyset=\infty$. For $0\leq t<\min\{t_r,t_{\max}\}$, \eqref{eq:nu_integral_decay} yields
\begin{equation}\label{eq:nu_nonincreasing}
    \nu(z(t))\leq\nu(z_0)<c.
\end{equation}
Since
\begin{equation*}
    \nu(x_p,w)\geq(1-d^\star)\vartheta(x_p) \geq(1-d^\star)r_1(\|x_p\|_{\R^n}),
\end{equation*}
we obtain
\begin{equation}\label{eq:xp_invariant_bound}
   \hspace{-1mm} \|x_p(t)\|_{\R^n} \leq \delta_c \coloneqq r_1^{-1}\left(\tfrac{\nu(z_0)}{1-d^\star}\right) <r_1^{-1}\left(\tfrac{c}{1-d^\star}\right)<\delta_r.
\end{equation}
If $t_r<t_{\max}$, continuity of $x_p$ and \eqref{eq:xp_invariant_bound} would imply simultaneously $\|x_p(t_r)\|_{\R^n}=\delta_r$ and $\|x_p(t_r)\|_{\R^n}\leq\delta_c<\delta_r$, which is impossible. Thus, the solution cannot leave $\calB_{\delta_r}$ through its boundary.

It remains to exclude finite-time blow-up. Recall $m_\rho=\tfrac{1-\rho\|\calA^{-1}\|}{2}>0$ from \eqref{eq:mrho_Mrho}.
By \eqref{eq:Lyap_function_1}, \eqref{eq:nu_nonincreasing}, and the definition
of $V_\rho$,
\begin{align}\label{eq:error_uniform_bound}
    d^\star m_\rho
    \|w(t)-w^\star(x_p(t))\|_{\calW}^2 \leq d^\star V_\rho(x_p(t),w(t))\leq\nu(z(t))\leq\nu(z_0).
\end{align}
The map $w^\star=\calS^{-1}\iota_3$ is Lipschitz, see Theorem \ref{thm:bijectivity} and  $\calS^{-1}(0)=0$, see Theorem \ref{thm:bijectivity}; hence, with $L_\calS$ denoting its Lipschitz constant,
\begin{equation*}
    \|w^\star(x_p(t))\|_{\calW}\leq L_\calS\delta_c.
\end{equation*}
Consequently, using the triangle inequality,
\begin{equation}\label{eq:w_uniform_bound}
    \|w(t)\|_{\calW} \leq \sqrt{\tfrac{\nu(z_0)}{d^\star m_\rho}}+L_\calS\delta_c
\end{equation}
for all $t<\min\{t_r,t_{\max}\}$. If $t_{\max}<\infty$ occurred before the exit time, \eqref{eq:xp_invariant_bound} and \eqref{eq:w_uniform_bound} would contradict the blow-up alternative
\eqref{eq:finite_time_blow_up}. We therefore have
\begin{equation*}
    t_{\max}=t_r=\infty.
\end{equation*}
In particular, $\Omega_c$ is positively invariant and every solution starting
in $\Omega_c$ is unique, global, and bounded.

We next prove convergence. Recall $M_\rho=\tfrac{1+\rho\|\calA^{-1}\|}{2}$ from \eqref{eq:mrho_Mrho} and set $e\coloneqq w-w^\star(x_p)$.
Moreover, by the quadratic lower bound on $\psi_p$, see Assumption \ref{ass:stabilizing_controller}, and the coercivity of $W_\rho$, see \eqref{eq:Lyap_function_1},
\begin{equation*}
    \psi_p(x_p)+V_\rho(x_p,w)
    \geq \underline c\|x_p\|_{\R^n}^2+m_\rho\|e\|_{\calW}^2.
\end{equation*}
Using $w^\star(0)=0$ and the Lipschitz continuity of
$w^\star$, we obtain
\begin{equation*}
     \|w\|_{\mathcal W} \leq \|e\|_{\mathcal W} +L_{\mathcal S}\|x_p\|_{\mathbb R^n}.
\end{equation*}
Consequently, for $z=(x_p, w) \in Z$:
\begin{align*}
    \|z\|_Z^2
    =
    \|x_p\|_{\mathbb R^n}^2+\|w\|_{\mathcal W}^2 \leq
    (1+2L_{\mathcal S}^2)\|x_p\|_{\mathbb R^n}^2
    +2\|e\|_{\mathcal W}^2 \leq \max\{1+2L_{\mathcal S}^2,2\}
    \bigl(
        \|x_p\|_{\mathbb R^n}^2+\|e\|_{\mathcal W}^2
    \bigr).
\end{align*}
Moreover,
\begin{align*}
    \psi_p(x_p)+V_\rho(x_p,w)
    \geq
    \underline c\|x_p\|_{\mathbb R^n}^2
    +m_\rho\|e\|_{\mathcal W}^2 \geq
    \min\{\underline c,m_\rho\}
    \bigl(
        \|x_p\|_{\mathbb R^n}^2+\|e\|_{\mathcal W}^2
    \bigr).
\end{align*}
Thus, with $q\coloneqq \tfrac{\min\{\underline c,m_\rho\}}{\max\{1+2L_{\mathcal S}^2,2\}}>0$ and $z=(x_p,w)\in\calB_\delta\times\calW$:
\begin{equation}\label{eq:dissipation_coercive_z}
    \psi_p(x_p)+V_\rho(x_p,w)
    \geq q\|z\|_Z^2.
\end{equation}
Since \(\delta_r<\delta\), the set \(\overline{\mathcal B}_{\delta_r}\) is a compact subset of \(\mathcal B_\delta\). The continuity of \(\nabla\vartheta\) therefore implies that
\begin{equation*}
     M_r\coloneqq \max_{\|x_p\|_{\mathbb R^n}\leq\delta_r} \|\nabla\vartheta(x_p)\|_{\mathbb R^n}<\infty.
\end{equation*}
Since \(\vartheta(0)=0\), the fundamental theorem of calculus
gives
\begin{equation*}
    \vartheta(x_p)
    = \int_0^1 \nabla\vartheta(sx_p)^\top x_p\,\mathrm ds \leq
    M_r\|x_p\|_{\mathbb R^n},\quad x_p\in\calB_{\delta_r}.
\end{equation*}
Thus, defining $\overline r(s)\coloneqq(M_r+1)s$ we have \(\overline r\in\mathcal K\) and
\begin{equation*}
     \vartheta(x_p) \leq\overline r(\|x_p\|_{\mathbb R^n}),  \qquad x_p\in\mathcal B_{\delta_r}.
\end{equation*}
Now, the upper bound in \eqref{eq:Lyap_function_1} and the Lipschitz continuity of $\calS^{-1}$ give
\begin{align*}
    \nu(z) \leq (1-d^\star)\overline r(\|x_p\|_{\R^n}) + d^\star M_\rho\|e\|_{\calW}^2
    \leq (1-d^\star)\overline r(\|z\|_Z) + d^\star M_\rho(1+L_\calS)^2\|z\|_Z^2\eqqcolon\alpha_2(\|z\|_Z),
\end{align*}
where $\alpha_2\in\calK$. Hence, for $z \in \Omega_c$:
\begin{equation*}
    \dot\nu_+(z) \leq -\lambda_\varepsilon q\|z\|_Z^2\leq-\lambda_\varepsilon q
      \bigl(\alpha_2^{-1}(\nu(z))\bigr)^2.
\end{equation*}
By \cite[Thm. 11.2.7 b.]{CurtainZwart2020}, for all $\varepsilon < \varepsilon^\star$:
\begin{equation*}
    z(t) = z(t;z_0)\to 0, \qquad \text{for all } z_0 \in \Omega_c.
\end{equation*}

For completeness, the same estimates also yield Lyapunov stability, see \cite[Def. 11.2.2]{CurtainZwart2020}. From $\|z\|_Z\leq(1+L_\calS)\|x_p\|_{\R^n}+\|e\|_{\calW}$ we infer that, whenever $\|z\|_Z\geq s$, at least one of the inequalities
\begin{equation*}
     \|x_p\|_{\mathbb R^n}
    \geq
    \tfrac{s}{2(1+L_{\mathcal S})},
    \qquad
    \|e\|_{\mathcal W}\geq\tfrac{s}{2}
\end{equation*}
must hold. Hence,
for sufficiently small $s>0$, define
\begin{equation*}
    \alpha_1(s)\coloneqq
    \min\left\{
       (1-d^\star)r_1\left(\tfrac{s}{2(1+L_\calS)}\right),
       \tfrac{d^\star m_\rho}{4}s^2
    \right\}.
\end{equation*}
It follows that
\begin{equation*}
    \alpha_1(\|z\|_Z)\leq\nu(z)\leq\alpha_2(\|z\|_Z)
\end{equation*}
in a neighborhood of the origin. Together with the monotonicity of $\nu$
along solutions, this proves Lyapunov stability.

Finally, choose any $\delta'>0$ such that
\begin{equation}\label{eq:delta_prime_choice}
    0<\delta'<  \min\left\{\delta_r,\alpha_2^{-1}(c)\right\}.
\end{equation}
Then: $ \calB_{\delta'}^Z
    \coloneqq\{z\in Z:\|z\|_Z<\delta'\}
    \subset\Omega_c$. Therefore, every solution starting in $\calB_{\delta'}^Z$ is unique, global, and bounded and converges to the origin. This proves local asymptotic
stability.

Finally, assume that $f_p$ is globally Lipschitz and that Assumption~\ref{ass:stabilizing_controller} holds with $\delta=\infty$ and $r_1\in\calK_\infty$. By \eqref{eq:orth_proj} and Theorem~\ref{thm:bijectivity}, the feedback
\begin{equation*}
    \kappa = \Pi_{[\underline u,\overline u]}\bigl(\tfrac1\alpha B_p^\top \iota_3^*\calS^{-1}\iota_3\,\cdot\,\bigr)
\end{equation*}
is globally Lipschitz with constant at most $\|B_p\|L_\calS/\alpha$, since $\Pi_{[\underline u,\overline u]}$ is nonexpansive and $\iota_3$ is an isometry. Hence $x_p \mapsto f_p(x_p)+B_p\kappa(x_p)$ is globally Lipschitz, so Lemma~\ref{lem:field_bound_ph} holds with $\delta=\infty$, and $f_\varepsilon$ is globally Lipschitz, so that the mild solution is global for every $z_0\in Z$. Consequently all estimates established above are valid on all of $Z$, and the constants $\eta,\gamma,\beta_1,\beta_2$, and therefore $\varepsilon^\star$, are independent of the initial value.

Now, let $\varepsilon < \varepsilon^\star$ and let $z_0=(x_p^0,w_0)\in Z$ be arbitrary. Since $r_1\in\calK_\infty$, we may choose the auxiliary radius $\delta_r>0$ so large that $ \delta_r > \|x_p^0\|_{\R^n}$ and $(1-d^\star)\,r_1(\delta_r)>\nu(z_0)$,
and subsequently $c\in\bigl(\nu(z_0),(1-d^\star)r_1(\delta_r)\bigr)$. With these choices $z_0\in\Omega_c$, and the argument of the previous paragraphs applies verbatim and yields $z(t;z_0)\to0$ as $t\to\infty$. Since $z_0\in Z$ was arbitrary and Lyapunov stability has already been established, the origin is globally asymptotically stable.
\end{proof}

\section{Numerical simulation}\label{sec:num}
\noindent
We illustrate the theoretical results with a genuinely nonlinear, undamped port-Hamiltonian \emph{Duffing oscillator}. Precisely, the nonlinear plant is given by 
\[
\dot q = p, 
\qquad
\dot p = - q - \beta q^3 + u_{\mathrm{p}},
\]
with initial condition \(x_{\mathrm{p}}(0)=(1.5,0)^\top\). 

Following \eqref{eq:OCP}, the underlying OCP reads 
\begin{align*}
\begin{split}
        \min_{(x,u)\in H^1([0,t_f],\R^2)\times L^2([0,t_f],\R)} \int_0^{t_f}\tfrac{1}{2}\|x(\tau)\|^2 + \tfrac{\alpha}{2} \|u(\tau)\|^2 \, \mathrm{d}\tau  \\ \mathrm{s.t.}\quad \dot x(\tau) = {\begin{smallpmatrix}
            0 & 1 \\ -1 & 0
        \end{smallpmatrix}} x(\tau) + {\stack{0}{1}} u(\tau),\quad x(0)=x_0, \qquad u(\tau) \in [-0.5,0.7] \text{ a.e..}
\end{split}
\end{align*}
Note that the dynamics in the OCP are the linearization of the Duffing plant about the origin, whereas the plant itself is nonlinear. The singularly perturbed plant-controller dynamics~\eqref{eq:sing_pert} generate the applied control. The control input is constrained to \(u_{\mathrm{p}}\in[-0.5,0.7]\). The prediction horizon is $t_f=5$, discretized with $20$ implicit-Euler steps.

Figure~\ref{fig:state-norms} shows the plant-state norm and the norm of the full interconnected state for \(\varepsilon\in\{10^{-3},10^{-2},10^{-1}\}\). Smaller values of \(\varepsilon\), corresponding to faster primal-dual dynamics, yield rapid convergence of both norms. For the largest value, \(\varepsilon=10^{-1}\), the optimizer evolves on a time scale comparable to that of the plant, resulting in substantially slower decay. Figure~\ref{fig:plant-control} shows the corresponding control inputs. The applied control initially reaches the constraint boundaries, demonstrating that the projection mechanism enforces admissibility. Once the state approaches the equilibrium, the control leaves the active set and converges smoothly toward zero; for $\varepsilon=10^{-1}$ this happens markedly later.

\begin{figure}[t]
    \centering
    \includegraphics[width=0.49\columnwidth]
        {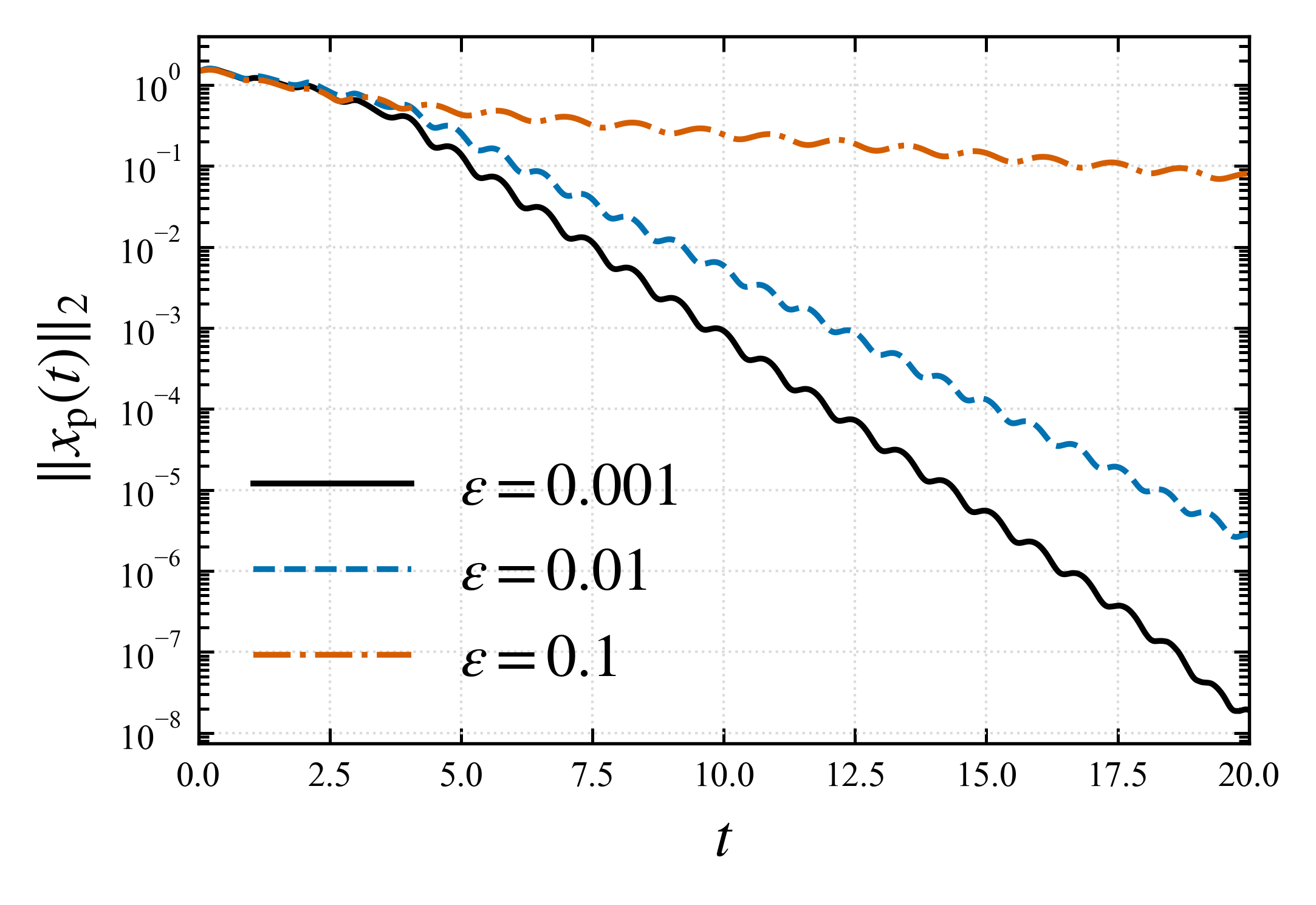}
    \includegraphics[width=0.49\columnwidth]
        {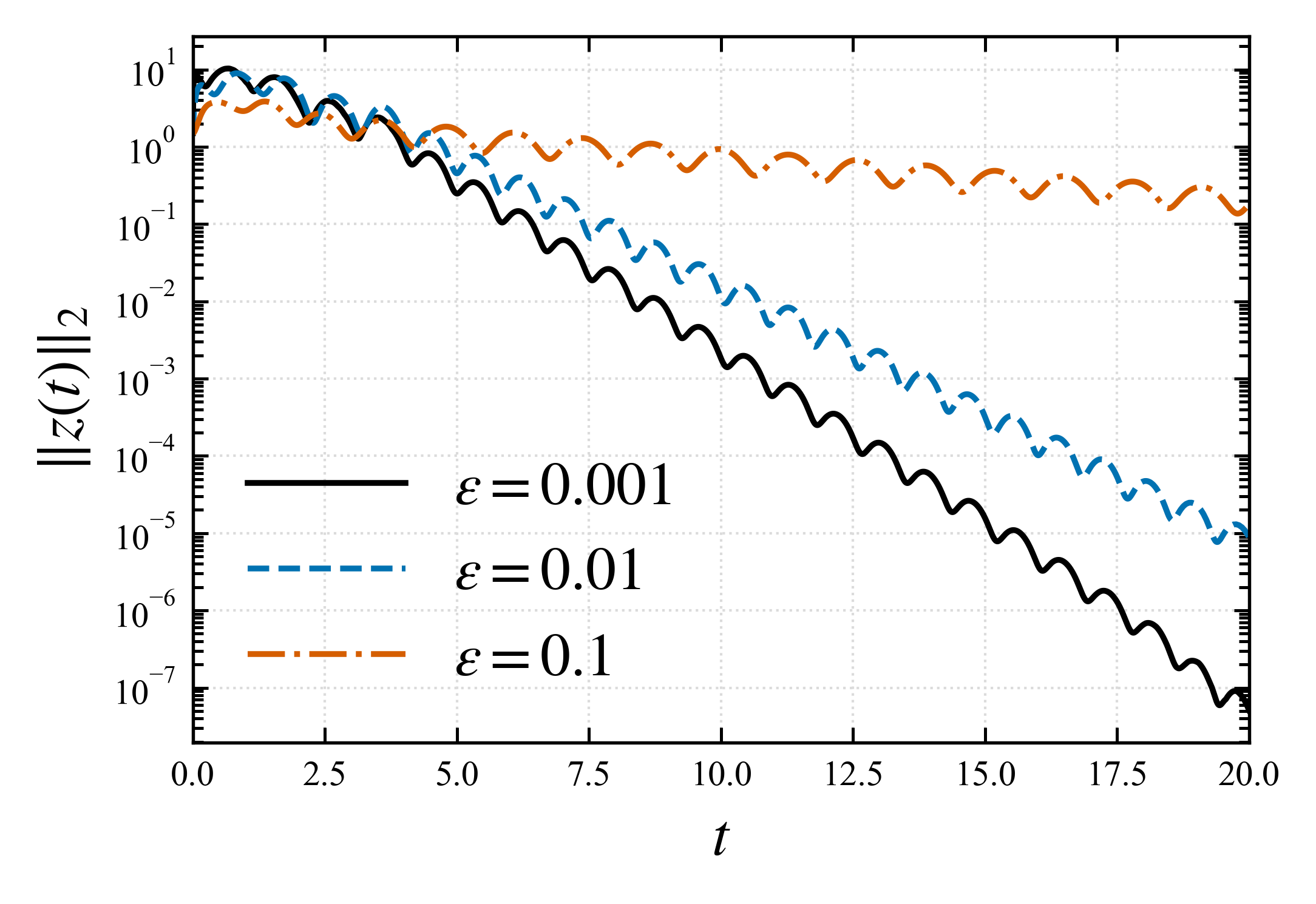}
    \caption{State-norm trajectories for different values of the time-scale parameter $\varepsilon$: plant-state norm and full interconnected-state norm.}
    \label{fig:state-norms}
\end{figure}
\begin{figure}[t]
    \centering
    \includegraphics[width=0.5\columnwidth]{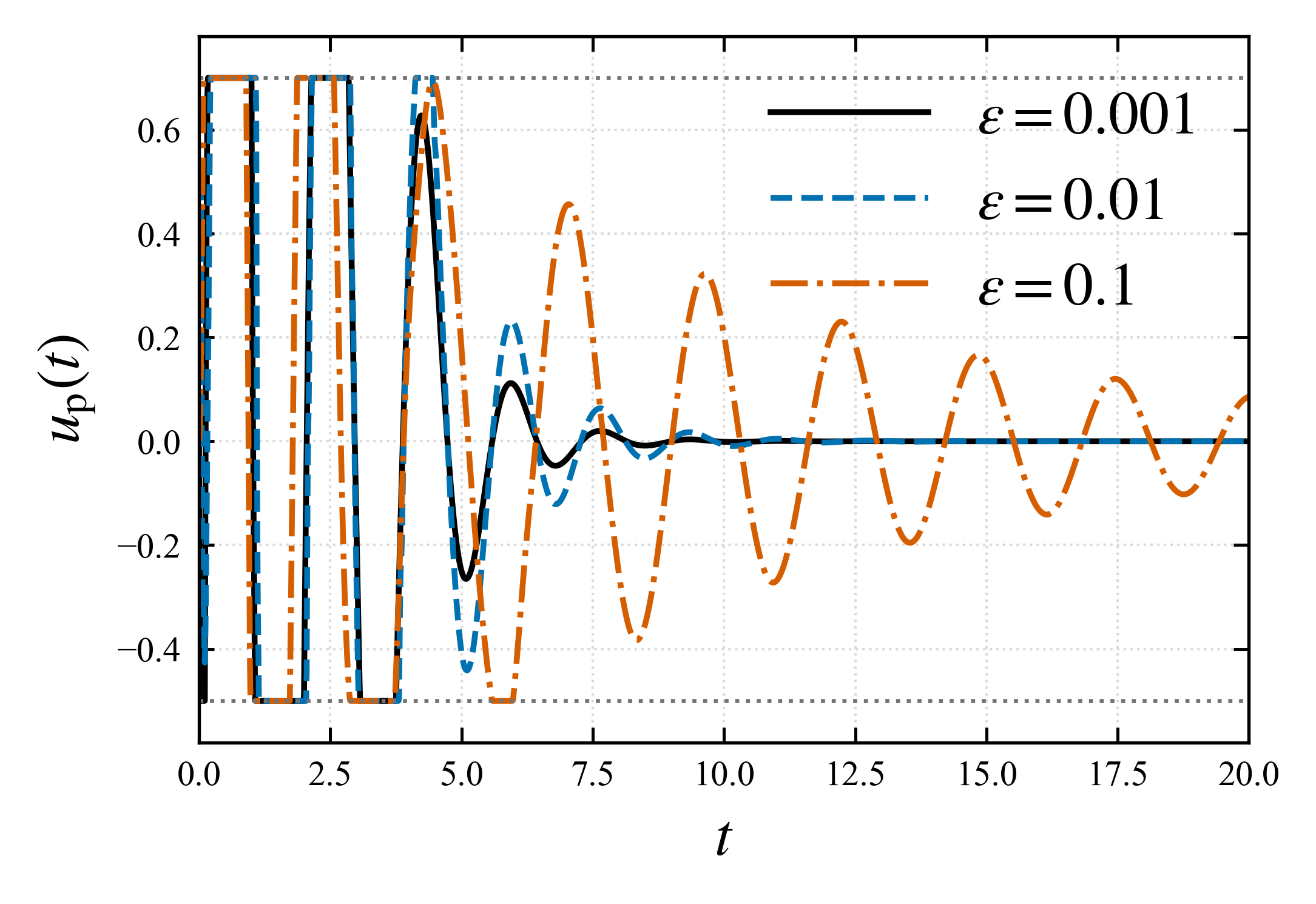}
    \caption{Control input applied to the nonlinear Duffing plant.}
    \label{fig:plant-control}
\end{figure}
\section{Conclusion}

\noindent We proposed an optimization-based control-by-interconnection framework for nonlinear pHs by representing finite-horizon primal–dual optimization dynamics as a pH controller. The coupled plant–controller dynamics form a singularly perturbed system, with the parameter $\varepsilon$ determining the optimizer speed. Using a composite Lyapunov function, we established local asymptotic stability for sufficiently small $\varepsilon$ and global stability under additional assumptions. Pointwise control constraints were incorporated through a projection operator, and the numerical example illustrated the effects of optimizer speed and input saturation.

\bibliography{references}

@article{cherukuri2017saddle,
  title={Saddle-point dynamics: conditions for asymptotic stability of saddle points},
  author={Cherukuri, Ashish and Gharesifard, Bahman and Cortes, Jorge},
  journal={SIAM Journal on Control and Optimization},
  volume={55},
  number={1},
  pages={486--511},
  year={2017},
  publisher={SIAM}
}

@article{gernandt2025port,
  title={Port-{H}amiltonian structures in infinite-dimensional optimal control: Primal--Dual gradient method and control-by-interconnection},
  author={Gernandt, Hannes and Schaller, Manuel},
  journal={Systems Control Lett.},
  volume={197},
  pages={106030},
  year={2025},
  publisher={Elsevier}
}

@book{Barb10,
  title={Nonlinear Differential Equations of Monotone Types in Banach Spaces},
  author={Barbu, V.},
  publisher={Springer, New York},
  series={Springer Monographs in Mathematics},
  year={2010},
 }

@book{BausComb2011,
  title={Convex Analysis and Monotone Operator Theory in Hilbert Spaces},
  author={Heinz H. Bauschke and Patrick L. Combettes},
  publisher={CMS Books in Mathematics},
  year={2011},
 }

@book{EkelTema99,
  title={Convex analysis and variational problems},
  author={Ekeland, Ivar and Temam, Roger},
  year={1999},
  publisher={SIAM Philadelphia}
}

@book{jacob2012linear,
	title={Linear port-{H}amiltonian systems on infinite-dimensional spaces},
	author={Jacob, Birgit and Zwart, Hans J},
	volume={223},
	year={2012},
	publisher={Springer Science \& Business Media}
}

@book{CurtainZwart2020,
  author    = {Ruth Curtain and Hans Zwart},
  title     = {Introduction to Infinite-Dimensional Systems Theory: A State-Space Approach},
  series    = {Texts in Applied Mathematics},
  publisher = {Springer},
  year      = {2020},
  edition   = {1}
  }

@article{StegPers15,
title = {Port-{H}amiltonian Formulation of the Gradient Method Applied to Smart Grids},
journal = {IFAC-PapersOnLine},
volume = {48},
number = {13},
pages = {13-18},
year = {2015},
issn = {2405-8963},
doi = {https://doi.org/10.1016/j.ifacol.2015.10.207}    ,
author = {Stegink,T.W. and De Persis, C. and {van der Schaft}, A.J.},
}

@Article{Willems1972a,
	author    = {Willems, Jan C},
	title     = {Dissipative dynamical systems part I: General Theory},
	journal   = {Archive for rational mechanics and analysis},
	year      = {1972},
	volume    = {45},
	number    = {5},
	pages     = {321-351},
	publisher = {Springer},
}

@Book{Brogliato07,
	author    = {Brogliato, Bernard and Lozano, Rogelio and Maschke, Bernhard and Egeland, Olav},
	publisher = {Springer},
	title     = {Dissipative Systems Analysis and Control},
	year      = {2007},
	volume    = {2},
	journal   = {Theory and Applications},
}

@InCollection{maschke1993port,
  author    = {Maschke, Bernhard and van der Schaft, Arjan},
  booktitle = {Nonlinear Control Systems Design},
  publisher = {Elsevier},
  title     = {{Port-controlled {H}amiltonian systems: modelling origins and systemtheoretic properties}},
  year      = {1992},
  pages     = {359--365},
}

@ARTICLE{YoshInou2019,
	author={Yoshida, Keisuke and Inoue, Masaki and Hatanaka, Takeshi},
	journal={IEEE Control Syst. Lett.},
	title={Instant {MPC} for Linear Systems and Dissipati\-vity-Based Stability Analysis},
	year={2019},
	volume={3},
	number={4},
	pages={811-816},
	doi={10.1109/LCSYS.2019.2918095}
}

@article{ortega2002interconnection,
  title={Interconnection and damping assignment passivity-based control of port-controlled {H}amiltonian systems},
  author={Ortega, Romeo and Van Der Schaft, Arjan and Maschke, Bernhard and Escobar, Gerardo},
  journal={Automatica},
  volume={38},
  number={4},
  pages={585--596},
  year={2002},
  publisher={Elsevier}
}

@book{RawlMayn17,
  title={Model {P}redictive {C}ontrol: {T}heory, {C}omputation, and {D}esign},
  author={Rawlings, James Blake and Mayne, David Q and Diehl, Moritz and others},
  volume={2},
  year={2017},
  publisher={Nob Hill Publishing Madison, WI}
}

@book{GrunPann16,
  title={Nonlinear model predictive control},
  author={Gr{\"u}ne, Lars and Pannek, J{\"u}rgen},
  booktitle={Nonlinear model predictive control: {T}heory and algorithms},
  year={2016},
  publisher={Springer}
}

@article{Pham22,
title = {A combined {C}ontrol by {I}nterconnection—{M}odel {P}redictive {C}ontrol design for constrained {P}ort-{H}amiltonian systems},
journal = {Systems Control Lett.},
volume = {167},
pages = {105336},
year = {2022},
issn = {0167-6911},
doi = {https://doi.org/10.1016/j.sysconle.2022.105336} ,
author = {T.H. Pham and N.M.T. Vu and I. Prodan and L. Lefèvre},
}

@inproceedings{vu2023port,
  title={Port-{H}amiltonian observer for state-feedback control design},
  author={Vu, NMT and Pham, TH and Prodan, I and Lefèvre, L},
  booktitle={2023 European Control Conference (ECC)},
  pages={1--6},
  year={2023},
  organization={IEEE}
}

@article{Zanelli21,
title = {A {L}yapunov function for the combined system-optimizer dynamics in inexact model predictive control},
journal = {Automatica},
volume = {134},
pages = {109901},
year = {2021},
issn = {0005-1098},
doi = {https://doi.org/10.1016/j.automatica.2021.109901} ,
author = {Andrea Zanelli and Quoc Tran-Dinh and Moritz Diehl},
}

@article{CamSch23,
author = {Camlibel, M. K. and van der Schaft, A. J.},
title = {Port-{H}amiltonian Systems Theory and Monotonicity},
journal = {SIAM J. Control Optim.},
volume = {61},
number = {4},
pages = {2193-2221},
year = {2023},
doi = {10.1137/22M1503749},
eprint = { 
        https://doi.org/10.1137/22M1503749
}
}

@article{SchJ14,
  title = {Port-{{{H}amiltonian Systems Theory}}: {{An Introductory Overview}}},
  author = {van der Schaft, A. J. and Jeltsema, D.},
  year = {2014},
  journal = {Foundations and Trends{\textregistered} in Systems and Control},
  volume = {1},
  number = {2-3},
  pages = {173--378},
  issn = {2325-6818},
  doi = {10.1561/2600000002}
}

@article{MehU22, title={Control of port-{H}amiltonian differential-algebraic systems and applications}, volume={32}, DOI={10.1017/S0962492922000083}, journal={Acta Numerica}, publisher={Cambridge University Press}, author={Mehrmann, V. and Unger, B.}, year={2023}, pages={395–515}}

@article{GeSc_DiehBock05,
	title={A real-time iteration scheme for nonlinear optimization in optimal feedback control},
	author={Diehl, Moritz and Bock, Hans Georg and Schl{\"o}der, Johannes P},
	journal={SIAM J. Control Optim.},
	volume={43},
	number={5},
	pages={1714--1736},
	year={2005},
	publisher={SIAM},
    doi = {10.1137/S0363012902400713}
}

@article{GeSc_ScokMayn99,
  author={Scokaert, P.O.M. and Mayne, D.Q. and Rawlings, J.B.},
  journal={IEEE Trans. Automat. Control}, 
  title={Suboptimal model predictive control (feasibility implies stability)}, 
  year={1999},
  volume={44},
  number={3},
  pages={648-654},
  doi={10.1109/9.751369}}

@article{CherMall2016,
	title = {Asymptotic convergence of constrained primal–dual dynamics},
	journal = {Systems Control Lett.},
	volume = {87},
	pages = {10-15},
	year = {2016},
	issn = {0167-6911},
	doi = {https://doi.org/10.1016/j.sysconle.2015.10.006},
	author = {Ashish Cherukuri and Enrique Mallada and Jorge Cortés},
}

@incollection{van2004port,
  title={Port-{H}amiltonian systems: network modeling and control of nonlinear physical systems},
  author={van der Schaft, Arjan J},
  booktitle={Advanced dynamics and control of structures and machines},
  pages={127--167},
  year={2004},
  publisher={Springer}
}

@book{khalil2002nonlinear,
author    = {Khalil, Hassan K.},
title     = {Nonlinear Systems},
edition   = {3},
year      = {2002},
publisher = {Prentice Hall},
address   = {Upper Saddle River, NJ},
isbn      = {0-13-067389-7}
}

@book{villani2009hypocoercivity,
  title={Hypocoercivity},
  author={Villani, C{\'e}dric},
  volume={202},
  year={2009},
  publisher={American Mathematical Society}
}

@article{achleitner2025hypocoercivity,
  title={Hypocoercivity in Hilbert spaces},
  author={Achleitner, Franz and Arnold, Anton and Mehrmann, Volker and Nigsch, Eduard A},
  journal={Journal of Functional Analysis},
  volume={288},
  number={2},
  pages={110691},
  year={2025},
  publisher={Elsevier}
}

@article{jakob2026generalized,
  title={A Generalized Plant Perspective on Linear-Convex Feedback Optimization},
  author={Jakob, Fabian and Iannelli, Andrea},
  journal={arXiv preprint arXiv:2606.14471},
  year={2026}
}

@article{esterhuizen26,
author = {Esterhuizen, Willem and Maschke, Bernhard and Preuster, Till and Schaller, Manuel and Worthmann, Karl},
title = {Nonlinear Controlled Port-{H}amiltonian Systems: Existence of (Optimal) Solutions},
journal = {SIAM J. Control Optim.},
volume = {64},
number = {4},
pages = {2569-2592},
year = {2026},
doi = {10.1137/24M1705093}
}
\bibliographystyle{plain}

\end{document}